\documentclass[11pt,reqno]{amsart}
\usepackage{latexsym,amsmath}
\usepackage[left=3cm,top=2.5cm,right=3cm,bottom=2.5cm]{geometry}
\usepackage{color}
\usepackage{amsthm}
\usepackage{amssymb}
\usepackage{epsfig}
\DeclareMathOperator{\ex}{ex}

\DeclareMathOperator{\forb}{Forb}
\DeclareMathOperator{\s}{sc}
\DeclareMathOperator{\SC}{SC}
\DeclareMathOperator{\KC}{KC}

\begin{document}

\allowdisplaybreaks[2]

\newtheorem{theorem}{Theorem}[section]
\newtheorem{cor}[theorem]{Corollary}
\newtheorem{lemma}[theorem]{Lemma}
\newtheorem{fact}[theorem]{Fact}
\newtheorem{property}[theorem]{Property}
\newtheorem{proposition}[theorem]{Proposition}
\newtheorem{claim}[theorem]{Claim}
\newtheorem{conjecture}[theorem]{Conjecture}
\newtheorem{definition}[theorem]{Definition}
\theoremstyle{definition}
\newtheorem{example}[theorem]{Example}
\newtheorem{remark}[theorem]{Remark}
\newcommand\eps{\varepsilon}
\newcommand\la {\lambda}
\newcommand{\E}{\mathbb E}
\newcommand{\Var}{{\rm Var}}
\newcommand{\Prob}{\mathbb{P}}
\newcommand{\N}{{\mathbb N}}
\newcommand{\eqn}[1]{(\ref{#1})}
\newcommand{\gP}{\mathcal{P}}
\newcommand{\gQ}{\mathcal{Q}}
\newcommand{\dist}{d}
\newcommand{\carlos}[1]{{\bf [~Carlos Oct 28:\ } {\em #1}{\bf~]}}

\def\dcup{\dot\cup}

\title{Hypergraphs with many Kneser colorings  (Extended version)}
\author[C. Hoppen]{Carlos Hoppen}
\address{Instituto de Matem\'atica, UFRGS -- Avenida Bento Gon\c{c}alves, 9500, 91501--970 Porto Alegre, RS, Brazil}
\email{choppen@ufrgs.br}
\thanks{The first author acknowledges the support of FAPERGS (Proc.~10/0388-2), FAPESP (Proc. 2007/56496-3) and CNPq (Proc.~484154/2010-9).}

\author[Y. Kohayakawa]{Yoshiharu Kohayakawa}
\address{Instituto de Matem\'{a}tica e Estat\'{i}stica, Universidade de S\~{a}o Paulo, Rua do Mat\~{a}o, 1010, CEP 05508-090, S\~{a}o Paulo, Brazil}
\email{yoshi@ime.usp.br}
\thanks{The second author was partially supported by CNPq (Proc.~308509/2007-2, 484154/2010-9).}

\author[H. Lefmann]{Hanno Lefmann}
\address{Fakult\"at f\"ur Informatik, Technische Universit\"at Chemnitz,
Stra\ss{}e der Nationen 62, D-09107 Chemnitz, Germany}
\email{Lefmann@Informatik.TU-Chemnitz.de}

\begin{abstract}
For fixed positive integers $r, k$ and $\ell$ with $1 \leq \ell < r$ and an $r$-uniform hypergraph $H$, let $\kappa (H, k,\ell)$ denote the number of
$k$-colorings of the set of hyperedges of $H$ for which any two
hyperedges in the same color class intersect in at least $\ell$ elements. Consider the function $\KC(n,r,k,\ell)=\max_{H\in{\mathcal H}_{n}} \kappa (H, k,\ell) $,
where the maximum runs over the family ${\mathcal H}_n$ of all $r$-uniform hypergraphs on $n$ vertices.  In this paper, we determine the asymptotic behavior of the function $\KC(n,r,k,\ell)$ for every fixed $r$, $k$ and $\ell$ and describe the extremal hypergraphs.
This variant of a problem of Erd\H{o}s and Rothschild, who considered
edge colorings of graphs without a monochromatic triangle,
is related to the Erd\H{o}s--Ko--Rado Theorem
on intersecting systems of sets [Intersection Theorems for
Systems of Finite Sets, Quarterly Journal of Mathematics, Oxford Series, Series 2,
{\bf 12} (1961), 313--320].

\end{abstract}

\maketitle

\section{Introduction}
We consider  $r$-uniform hypergraphs $H=(V,E)$. A
\emph{hypergraph} $H =(V,E)$ is given by its \emph{vertex set} $V$ and
its set $E$ of \emph{hyperedges}, where $e \subseteq V$ for each hyperedge
$e \in E$, and $H =(V, E)$ is said to be \emph{$r$-uniform} if each hyperedge $e \in E$ has cardinality $r$. For a
fixed  $r$-uniform hypergraph $F$, an $r$-uniform ``host-hypergraph''
$H$ and an integer $k$, let $c_{k,F}(H)$
denote the number of
$k$-colorings of the set of hyperedges
of $H$ with no monochromatic copy of $F$ and
let $c_{k,F}(n)=\max \{c_{k,F}(H) \colon H \in {\mathcal H}_n\}$,
where  ${\mathcal H}_n$ is the family of
all $r$-uniform hypergraphs on $n$ vertices. For instance, if $H$ is a graph
and $F$ is a path of length two, then each color class has to be a matching and
$c_{k,F}(H)$ is the number of proper $k$-edge colorings
of $H$. Moreover, given an $r$-uniform
hypergraph $F$, let $\ex(n,F)$ be the usual
\emph{Tur\'an number} for $F$, i.e., the maximum
number of hyperedges in an $r$-uniform $n$-vertex hypergraph that contains no copy of $F$. A hypergraph for which maximality is achieved is said to be an \emph{extremal hypergraph for $\ex(n,F)$}.

Every coloring of the set of hyperedges of
any extremal hypergraph $H$ for $\ex(n,F)$ trivially contains no monochromatic copy of $F$ and,
hence,
$
c_{k,F}(n)\geq k^{\ex(n,F)}
$
for all $k\geq 2$. On the other hand, if $\forb_F(n)$ denotes the family of all
hypergraphs with vertex set $[n]=\{1,\ldots,n\}$ that contain no copy of $F$, every $2$-coloring of the set of hyperedges of a hypergraph $H$ on $[n]$ containing no monochromatic copy of $F$ gives rise
to a member of $\forb_F(n)$; thus
$
c_{2,F}(n) \leq |\forb_F(n)|
$.
The size of $\forb_F(n)$ was first studied by Erd\H{o}s,
Kleitman, and Rothschild~\cite{ErKlRo76}
for $F=K_3$, the triangle.
This has been extended  by Kolaitis, Pr\"omel,
and Rothschild~\cite{KoPrRo85,KoPrRo87} to the case when $F=K_{\ell}$ is a clique on $\ell$ vertices. For an arbitrary graph $F$,
Erd\H{o}s, Frankl, and R\"odl~\cite{ErFrRo86} proved
the upper bound
$
|\forb_F(n)|\leq 2^{\ex(n,F)+o(n^2)};
$
see also~\cite{BaBoSi04,BaBoSi08}.
The results from~\cite{ErFrRo86} have been further
extended by Nagle, R\"odl, and
Schacht~\cite{NaRo01,NaRoSch06}
to $r$-uniform hypergraphs, namely
$
    |\forb_F(n)|\leq 2^{\ex(n,F)+o(n^r)}
$. Thus,
for any $r$-uniform hypergraph $F$ we have
\begin{equation}\label{eq:simple}
2^{\ex(n,F)}\leq c_{2,F}(n)\leq 2^{\ex(n,F)+o(n^r)}.
\end{equation}

For $r=2$ and cliques $F=K_{t}$, Yuster~\cite{Yu96} for $t = 3$
and Alon, Balogh, Keevash, and Sudakov~\cite{ABKS04} for any fixed $t \geq 3$
showed that the lower bound in~\eqref{eq:simple} is the
exact value of $c_{2,K_{t}}(n)$ for $n$ sufficiently large,
as conjectured by Erd\H{o}s and Rothschild (see \cite{Er74}).
Moreover, for $3$-colorings,
Alon, Balogh, Keevash, and Sudakov~\cite{ABKS04} proved
that $c_{3,K_{t}}(n)=3^{\ex(n,K_{t})}$ for $n$ sufficiently
large. In both cases,
$k=2$ and $k= 3$, equality is achieved only
by the $(t-1)$-partite Tur\'an graph on $n$ vertices. However, it was
observed in~\cite{ABKS04} that $ c_{k,K_{t}}(n)\gg k^{\ex(n,K_t)}$ for any fixed $k\geq 4$ as $n$ tends to infinity. Very recently, Pikhurko and  Yilma \cite{PYxx} succeeded in describing those graphs on $n$ vertices that achieve  $ c_{4,K_{3}}(n)$ as well as $ c_{4,K_{4}}(n)$. However, for $k \geq 5$ colors, or $k=4$ colors and forbidden complete graphs $K_{\ell}$, $\ell > 4$,  the extremal graphs are not known. 

An extension of these results to hypergraphs has been given recently in
\cite{LPRS}  for the Fano plane $F$,
the unique $3$-uniform hypergraph
with seven hyperedges on seven vertices where
every pair of distinct vertices is contained in
exactly one hyperedge. Fundamental in this direction was the determination of the Tur\'{a}n number
$\ex(n,F)
=\binom{n}{3}-\binom{\lceil n/2\rceil}{3}-\binom{\lfloor n/2\rfloor}{3}
$,
achieved by Keevash and Sudakov \cite{KS05} and F\"uredi and Simonovits
\cite{FS05}. The results in \cite{LPRS}
show that for the Fano plane $F$, for $n$ sufficiently 
large,
in the case of $k$-colorings, $k \in \{2,3\}$,
every $3$-uniform hypergraph $H$ on $n$ vertices satisfies
$
c_{k,F}(H)\le k^{\ex(n,F)}
$.
Moreover, equality is attained by the unique extremal hypergraph for $\ex(n,F)$. Also, for fixed $k \geq 4$, the inequality $c_{k,F}(n)\gg k^{\ex(n,F)} $ holds as $n$ tends to infinity. Very recently, a similar phenomenon has been proved to hold in several other instances, see for example \cite{LP11} and \cite{LPS11}.

Here, we investigate a variant of the original problem of Erd\H{o}s
and Rothschild, where we forbid pairs of hyperedges of the same color
that share fewer than $\ell$ vertices, \emph{thus forcing every color class
to be $\ell$-intersecting}. Formally, for fixed integers $\ell, r$ with
$1 \leq \ell < r$, and $i \in \{0,\ldots,\ell-1\}$, let $F_{r,i}$ be the $r$-uniform hypergraph on $2r-i$ vertices with two hyperedges sharing exactly $i$ vertices, and let $\mathcal{B}_{r,\ell}=\{F_{r,i}:~i=0,\ldots,\ell-1\}$.
Following the notation above,
$c_{k,\mathcal{B}_{r,\ell}}(H)$
is the number of
$k$-colorings of the set of hyperedges of a hypergraph $H$ with no monochromatic copy of any $F \in \mathcal{B}_{r,\ell}$. Let $c_{k,\mathcal{B}_{r,\ell}}(n)=
\max\{ c_{k,
\mathcal{B}_{r,\ell}}(H) \colon H \in {\mathcal H}_n\}$, and
 set
$\KC(n,r,k,\ell)= c_{k,\mathcal{B}_{r,\ell}}(n)$ as well as
$\kappa (H, k,\ell) = c_{k,\mathcal{B}_{r,\ell}}(H)$. These $\mathcal{B}_{r,\ell}$-avoiding colorings with $k$ colors are called \emph{$(k,\ell)$-Kneser colorings}; as is well known, Lov\'{a}sz~\cite{lovasz} proved a conjecture of Kneser asserting that $c_{k,\mathcal{B}_{r,\ell}}(K_n^{(r)})=0$ when $\ell=1$, $n \geq k+ 2r -1$ and $K_n^{(r)}$ is the complete, $r$-uniform hypergraph on $n$ vertices. For developments in this line of research, see Ziegler~\cite{ziegler} and the references therein.

Our main concern here is to investigate which $n$-vertex $r$-uniform 
hypergraphs $H$ maximize $\kappa (H, k,\ell)$.
As one would expect, this problem is related to the well-known Erd\H{o}s--Ko--Rado Theorem~\cite{EKR}. Recall that, for $n$ large, the unique extremal hypergraph for $\ex(n,\mathcal{B}_{r,\ell})$ is the hypergraph on $n$ vertices whose hyperedges are all $r$-element subsets of $[n]$ containing a fixed $\ell$-element set. In line with the results in~\cite{ABKS04}, we obtain the following when colorings with two or three colors are considered. 
\begin{theorem} \label{otto2}
If $n \geq r > \ell$ are positive integers, then
\begin{equation}\label{frida4}
\KC(n,r,2,\ell) = 2^{\ex(n,\mathcal{B}_{r,\ell})}.
\end{equation}
Every $r$-uniform hypergraph $H$ on $[n]$ that is extremal for $\ex(n,\mathcal{B}_{r,\ell})$ achieves $\kappa(H,2,\ell)=\KC(n,r,2,\ell)$, and unless $\ell=1$ and $n=2r$, these are the single hypergraphs that achieve equality.
\end{theorem}

\begin{theorem}\label{thm_2and3}
For all positive integers $r$ and $\ell$, there exists $n_0>0$ such that, for $n>n_0$, 
\begin{equation}\label{frida2}
\KC(n,r,3,\ell) = 3^{\ex(n,\mathcal{B}_{r,\ell})}.
\end{equation}
Moreover, for $n>n_0$, the $r$-uniform hypergraphs $H$ achieving equality in \eqref{frida2} correspond to the extremal configurations for $\ex(n,\mathcal{B}_{r,\ell})$.
\end{theorem}

In the case of arbitrary $k \geq 4$, we obtain the asymptotic behavior of  $\KC(n,r,k,\ell)$ as $n$ tends to infinity for $r$ and $\ell$ fixed with $\ell < r$, and we describe the extremal hypergraphs. The following definition is important for our purposes.

\begin{definition}\label{def_extremal}
For integers $k,r \geq 2$, $1 \leq \ell <r$, $c \geq 1$ and
$n \geq \max \{r,c \ell\}$, let $C$ be a set of cardinality $c$ whose
elements are $\ell$-subsets of $[n]= \{ 1, \ldots , n\}$.
The \emph{$(C,r)$-complete}
\emph{hypergraph $H_{C,r}(n)$} is the hypergraph with vertex set $[n]$ whose hyperedges are all the $r$-subsets of $[n]$
containing some element of $C$ as a subset. If $C$ is a set of $
c(k)=\lceil k/3 \rceil$ mutually disjoint $\ell$-sets,
then the hypergraph $H_{C,r}(n)$ is denoted by
$H_{n,r,k,\ell}$.
\end{definition}

One of the main results in our work is that the hypergraph $H_{n,r,k,\ell}$ is always asymptotically close to being optimal.
\begin{theorem}\label{Tint1}
Let $r \geq 2$, $k \geq 2$ and $1 \leq \ell < r$ be fixed integers. Then
$$\KC(n,r,k,\ell)=(1+o(1)) \cdot
\kappa(H_{n,r,k,\ell},k,\ell),$$ where $o(1)$ is
a function that tends to $0$ as $n$ tends to infinity.
\end{theorem}
In spite of Theorem~\ref{Tint1}, it turns out that $H_{n,r,k,\ell}$ is \emph{not} extremal when either $k=4$ and $\ell>1$ or $k \geq 5$ and $r<2\ell-1$.  For this and related comments, see Theorem~\ref{Tint2}, Theorem~\ref{Tint3}(ii) and (iii), and Section~\ref{sec_final}.

It will be evident in the proof of Theorem \ref{Tint1} that the quest for the
asymptotic value of $\KC(n,r,k,\ell)$ and the characterization of the extremal
hypergraphs are strongly intertwined. As a matter of fact, we
focus on two special classes of Kneser colorings, which we prove to contain all but a negligible fraction of all Kneser colorings. On the one hand, the structure of the colorings in such classes leads to a series of symmetry properties of the extremal hypergraphs. On the other hand, these properties allow us to estimate accurately the number of Kneser colorings in each such special class, leading to the desired asymptotic value. More precisely, we fully describe the hypergraphs that are optimal for sufficiently large $n$ by making use of the following somewhat cumbersome definition.
\begin{definition}\label{optimal_hyper}
Fix integers $n$, $r\geq 2$, $k \geq 2$ and $1 \leq \ell < r$. The family of \emph{candidate hypergraphs} $\mathcal{H}_{r,k,\ell}(n)$ consists of all $n$-vertex $r$-uniform hypergraphs $H$ defined as follows.
\begin{itemize}
\item[(a)] If $k \in \{2,3\}$ or if $k \geq 5$ and $r \geq 2\ell-1$, then $H$ is isomorphic to $H_{n,r,k,\ell}$.

\item[(b)] If $k=4$, then $H$ is $H_{C,r}(n)$ for  $C=\{t_1,t_2\}$ with $|t_1 \cap t_2| = \ell - 1$, 
where the sets $t_i$ are $\ell$-subsets of the vertex set.

\item[(c)] If $k \geq 5$ and $r<2 \ell-1$, then $H$ is $H_{C,r}(n)$ for $C=\{t_1,\ldots,t_{c(k)}\}$, and each $t_i$ is an $\ell$-subset of the vertex set and 
$|t_i \cup t_j|>r$, for all $1\leq i <j \leq c(k)=\left\lceil k/3 \right\rceil$.
\end{itemize}
\end{definition}
Note that, if $r$, $k$ and $\ell$ are as in (a) and (b), the family $\mathcal{H}_{r,k,\ell}(n)$ contains a single hypergraph up to isomorphism.
\begin{theorem}\label{Tint2}
Given $r$, $k$ and $\ell$, there is $n_0>0$ such that,
for $n>n_0$,  if
$$\kappa(H,k,\ell)=\KC(n,r,k,\ell),$$
then $H \in \mathcal{H}_{r,k,\ell}(n)$.
\end{theorem}
Theorem~\ref{Tint2} immediately implies that, for $n$ sufficiently large, the
extremal hypergraph is unique when either $k=4$ or $k \geq 5$
and $r \geq 2\ell-1$, as $\mathcal{H}_{r,k,\ell}(n)$ contains a single hypergraph up to isomorphism. In particular, given any positive integer $k$,
the problem of finding the hypergraphs with most $(k,1)$- as well as 
$(k,2)$-Kneser colorings is completely solved for $n$ sufficiently large.

Moreover, if $r<2\ell-1$, let $\mathcal{C}_{k,\ell}$ be the family of set systems $C$ given in item (c) of Definition~\ref{optimal_hyper}. Theorem~\ref{Tint2} then tells us that, for $n \geq n_0$,
$$\KC(n,r,k,\ell)=\max \{\kappa(H_{C,r}(n),k,\ell)\colon C \in \mathcal{C}_{k,\ell}\}.$$

We actually use our work in the proof of Theorem~\ref{Tint2} to prove a stronger result, namely that, for $k \neq 4$, the set of hypergraphs $\mathcal{H}_{r,k,\ell}(n)$ is precisely the set of all hypergraphs that are asymptotically close to being extremal.
\begin{theorem}\label{Tint2_cont}
Let $k \neq 4$, $r$ and $\ell$ be fixed. For every $\eps>0$, there is $n_0>0$ such that,
for any $n>n_0$ and $H \in \mathcal{H}_{r,k,\ell}(n)$,
\begin{equation}\label{eqnova}
\displaystyle{\kappa(H,k,\ell) \geq (1-\eps)\KC(n,r,k,\ell)}.
\end{equation}
Conversely, there exist $n_1>0$ and $\eps_0>0$
such that, if $n>n_1$ and $\kappa(H,k,\ell) \geq (1-\eps_0)\KC(n,r,k,\ell)$,
then $H \in \mathcal{H}_{r,k,\ell}(n)$.
\end{theorem}
On the other hand, for $k=4$ and $\ell>1$, the situation is different. Recall that $\mathcal{H}_{r,4,\ell}(n)$ contains a single hypergraph up to isomorphism. However, we shall see that the set of hypergraphs $H$ that satisfy \eqref{eqnova} is larger. To the best of our knowledge, proving the existence of a unique extremal configuration for a problem with a large family of distinct asymptotically extremal configurations is rather uncommon. This is addressed in Theorem~\ref{theo1_inter}.  

The following is the analogue of Theorem~\ref{Tint2} for $k=4$.
\begin{definition}\label{defH4}
Fix integers $n,r \geq 2$ and $1 \leq \ell <r$. The family $\mathcal{H}^{\ast}_{r,4,\ell}(n)$ consists of all $n$-vertex $r$-uniform hypergraphs $H$ such that $H=H_{C,r}$ for $C=\{t_1,t_2\}$, where the distinct sets $t_i$ are $\ell$-subsets of the vertex set.
\end{definition}

\begin{theorem}\label{Tint2_cont4}
Let $r$ and $\ell$ be fixed. For every $\eps>0$, there is $n_0>0$ such that,
for any $n>n_0$ and $H \in \mathcal{H}^{\ast}_{r,4,\ell}(n)$,
$$\displaystyle{\kappa(H,4,\ell) \geq (1-\eps)\KC(n,r,4,\ell)}.$$
Conversely, there exist $n_1>0$ and $\eps_0>0$
such that, if $n>n_1$ and $\kappa(H,4,\ell) \geq (1-\eps_0)\KC(n,r,4,\ell)$,
then $H \in \mathcal{H}^{\ast}_{r,4,\ell}(n)$.
\end{theorem}

Theorems~\ref{Tint2_cont} and \ref{Tint2_cont4} may be naturally interpreted in terms of `stability' as in Simonovits's Stability Theorem~\cite{simonovits} for graphs. Roughly speaking, the problem of maximizing a function $f$ over a class of combinatorial objects $\mathcal{C}$ is said to be \emph{stable} if every object that is very close to maximizing $f$ is almost equal to the object that maximizes $f$. In our framework, this idea can be formalized as follows. Here, for two sets $A$ and $B$, we write $A \bigtriangleup B$ for their symmetric difference $(A\setminus B) \cup (B \setminus A)$.
\begin{definition}\label{def_stable}
Let $r$, $k$ and $\ell$ be fixed. The problem $\mathcal{P}_{n,r,k,\ell}$ of determining $\KC(n,r,k,\ell)$ is \emph{stable} if, for every $\eps>0$, there exist $\delta>0$ and $n_0>0$ such that the following is satisfied. Let $H^{\ast}$ be an $r$-uniform extremal hypergraph for $\KC(n,r,k,\ell)$, where $n>n_0$, and let $H$ be an $r$-uniform hypergraph on $[n]$ satisfying $\kappa(H,k,\ell) > (1-\delta) \KC(n,r,k,\ell)$. Then 
 $|E(H) \bigtriangleup E(H')|< \eps |E(H')|$ for some hypergraph $H'$ isomorphic to $H^{\ast}$.  
\end{definition}

Combining Theorems~\ref{Tint2_cont} and \ref{Tint2_cont4} with our work for $k \in \{2,3\}$, we may deduce exactly when $\mathcal{P}_{n,r,k,\ell}$ is stable.
\begin{theorem}\label{Tint3}
Let $k \geq 2$, $r$, and $\ell$ be positive integers with $\ell < r$.
\begin{itemize}
\item[(i)] If $k \in \{2,3\}$, then $\mathcal{P}_{n,r,k,\ell}$ is stable.

\item[(ii)] If $k=4$, then $\mathcal{P}_{n,r,k,\ell}$ is stable if and only if $\ell=1$.

\item[(iii)] If $k \geq 5$, then $\mathcal{P}_{n,r,k,\ell}$ is stable if and only if $r \geq 2 \ell-1$.
\end{itemize}
\end{theorem}
This instability result suggests that, when $\ell>1$, the precise determination of the extremal hypergraphs for $k=4$ and for $k \geq 5$ with $r < 2 \ell-1$ requires a very careful counting of the number of Kneser colorings of each of the candidate extremal hypergraphs. As it turns out, we were able to carry out the calculations for the case $k=4$ and $\ell>1$ (see~Theorem~\ref{theo1_inter}). 

The remainder of the paper is organized as follows. In Section \ref{sec_prem} we provide basic definitions and results, and we address the case $k=2$. Section \ref{sec_upper} is concerned with basic structural aspects of Kneser colorings, which, for $n$ sufficiently large, lead to the determination of $\KC(n,r,3,\ell)$ and of auxiliary upper bounds on $\KC(n,r,k,\ell)$ when $k \geq 4$. Additional properties of extremal hypergraphs are obtained in Section \ref{sec_extremal}, which are then used in Section \ref{sec_asy} to find an asymptotic formula for $\KC(n,r,k,\ell)$ when $k \geq 4$. Concluding remarks follow in Section \ref{sec_final}.

\section{Preliminaries}\label{sec_prem}

In this section, we consider Kneser colorings with two colors. Moreover, we introduce an optimization problem that plays an important role in the study of Kneser colorings with more colors. We start by formally stating our concepts and terminology.
\begin{definition}
An \emph{$r$-subset} ($r$-set) of a set $X$ is an $r$-element subset of $X$. For a positive integer $\ell$, we say that a family $F$ of sets is \emph{$\ell$-intersecting} if the intersection of any two sets in $F$ contains at least $\ell$ elements.
\end{definition}

\begin{definition}
A \emph{$(k,\ell)$-Kneser coloring} of a hypergraph $H=(V,E)$ is a
function $\Delta \colon E \longrightarrow [k]$ associating
a \emph{color} with each hyperedge with the property that any two hyperedges
with the same color are $\ell$-intersecting. A hypergraph admitting a $(k,\ell)$-Kneser coloring is called \emph{$(k,\ell)$-Kneser colorable} ($(k,\ell)$-colorable, for short), and the number of $(k,\ell)$-Kneser colorings of a hypergraph $H$ is denoted by $\kappa(H,k,\ell)$. Given positive integers $n$, $r$, $k$ and $\ell$, we define
$$\KC(n,r,k,\ell)=\max\{\kappa(H,k,\ell):H \textrm{ is an
$r$-uniform hypergraph on $n$ vertices}\};$$
that is, $\KC(n,r,k,\ell)$ is the maximum number of $(k,\ell)$-Kneser colorings on an $r$-uniform hypergraph on $n$ vertices.
\end{definition}

Recall from the introduction that $F_{r,i}$ is the $r$-uniform hypergraph on
$2r-i$ vertices with two hyperedges sharing exactly $i$ vertices and
$\mathcal{B}_{r,\ell}=\{F_{r,i}:~i=0,\ldots,\ell-1\}$. Moreover, the Tur\'{a}n number $\ex(n,\mathcal{B}_{r,\ell})$ is the largest number of hyperedges in a $\mathcal{B}_{r,\ell}$-free $r$-uniform hypergraph on $[n]$. The following result was proved by Ahlswede and Khachatrian \cite{AK}, generalizing the Erd\H{o}s--Ko--Rado Theorem~\cite{EKR}. It settles the problem of determining $\ex(n,\mathcal{B}_{r,\ell})$ and the associated extremal hypergraphs. Note that $\ex(n,\mathcal{B}_{r,\ell})=
\binom{n}{r}$ for $n \leq 2r - \ell$, as any two $r$-subsets of $[n]$ are $\ell$-intersecting. In what follows, $[n]^r$ denotes the set of all $r$-subsets of $[n]$.

\begin{theorem}[Ahlswede and Khachatrian~\cite{AK}] \label{otto1}
Let $n \geq r \geq \ell$ be positive integers with $(r-\ell+1)\left(2+\frac{\ell-1}{s+1}\right)< n < (r-\ell+1)\left(2+\frac{\ell-1}{s}\right)$ for some
non-negative integer $s \leq r-\ell$. Then
$$\ex(n,\mathcal{B}_{r,\ell})=|\mathcal{F}_s|=\left|\left\{F \in [n]^r\, :~|F \cap [1,\ell+2s]|\geq \ell+s\right\}\right|,$$
and $\mathcal{F}_s$ is, up to permutations, the unique optimum. (By convention, $\frac{a}{0}=\infty$.) If $n=(r-\ell+1)\left(2+\frac{\ell-1}{s+1}\right)$ for some non-negative integer $s < r - \ell$, we have $\ex(n,\mathcal{B}_{r,\ell})=|\mathcal{F}_s|=|\mathcal{F}_{s+1}|$ and, unless $\ell = 1$ and $n=2r$, any optimal system equals, up to permutations, either $\mathcal{F}_s$ or $\mathcal{F}_{s+1}$. If $s = r-\ell$, then $\ex( \mathcal{B}_{r, \ell}) =|{\mathcal F}_{r - \ell}|$ and any optimal system is equal to $\mathcal{F}_{r-\ell}$ up to permutations.
If $\ell=1$ and $n=2r$, an optimal system ${\mathcal F}$ may be built in such a way that,
 for every $r$-subset $A$ in $[n]$, either $A$ or its complement lies in
${\mathcal F}$.
\end{theorem}

The following property of the set systems $\mathcal{F}_s$ defined in the
statement of Theorem \ref{otto1} is particularly useful.
\begin{lemma}\label{auxlem1}
Let $r$ and $\ell$ be positive integers satisfying $\ell<r$. Consider a positive integer $n$, with the additional restriction $n>2r$ if $\ell=1$,
and a non-negative integer $s$ and the set system $\mathcal{F}_s$ corresponding to an extremal configuration for $\ex(n,\mathcal{B}_{r,\ell})$ defined in Theorem \ref{otto1}. If $e$ is an $r$-subset of $[n]$ that is not $\ell$-intersecting with an element of $\mathcal{F}_s$, then it is not $\ell$-intersecting with at least two elements of $\mathcal{F}_s$.
\end{lemma}

\begin{proof} We first consider the case $\ell=1$. From Theorem \ref{otto1},
the constant $s$ must have value $0$, while $n>2r$ by hypothesis.
In particular, $e$ does not contain $1$, whereas every
element $f \in \mathcal{F}_0$ contains $1$. Since  $n>2r$
 there are at least $r$ elements in $[n]$ disjoint from $e \cup \{1\}$,
hence we may define at least $\binom{r}{r-1}=r>1$ $r$-sets
in $\mathcal{F}_0$ that are disjoint from $e$.    

We now assume that $\ell>1$. Clearly, $s \leq r-\ell$ in this case, and Theorem \ref{otto1} implies $$n \geq (r-\ell+1)\left(2+\frac{\ell-1}{s+1}\right) \geq (r-\ell+1)\left(2+\frac{\ell-1}{r-\ell+1}\right)=2r-\ell+1.$$
Let $a=|e \cap [\ell+2s]|$. From the $\ell+2s-a$ elements in $[\ell+2s]\setminus e$, we choose either $\ell+2s-a$ or $r$, whichever is smaller.

If $r$ elements have been chosen, we are done, as we obtained an element $f$
of $\mathcal{F}_s$ that is both fully contained in $[\ell+2s]$,
hence $\ell+s \leq r \leq \ell+2s$, and disjoint from $e$.
Note that at least one of these inequalities is strict, as the converse would imply $s=0$ and $r=\ell$, contradicting our hypothesis. If $\ell+s < r$, the substitution of any element of $f$ by an element of $e$ yields an element of $\mathcal{F}_s$ that is not $\ell$-intersecting with $e$, since $\ell>1$. If $r < \ell+2s$, a second element of $\mathcal{F}_s$ whose intersection with $e$ has size at most one may be built through the substitution of any element of $f$ by an element of $[\ell+2s] \setminus f$.

Therefore we assume that $\ell+2s-a<r$. Keep in mind that we are building elements $g \in \mathcal{F}_s$ that are not $\ell$-intersecting with $e$ and that the $\ell+2s-a$ elements in $[\ell+2s] \setminus e$ have been added to $g$.  There are two cases, according to the relative order of $\ell+2s-a$ and $\ell+s$.

If $\ell+2s-a \geq \ell+s$, we add elements of $e$ to $g$ until $g$ has $r$ elements or $|g \cap e| = \ell-1$. It is clear that this addition can be done in more than one way, as at least one element has to be added, but clearly fewer than $|e|=r>1$ can be added.  At this point, either any such $g$ is an $r$-set, in which case we are done, or $r-|g|=r-2\ell-2s+a+1 \geq 1$. The number of elements of $[n]$ that are neither in $[\ell+2s]$ nor in $e$ is given by $b=n-(\ell+2s)-(r-a)$. The inequality $n \geq 2r -\ell+1$ leads to $b \geq r-2\ell-2s+a+1$, so that $g$ may be extended to an $r$-set without affecting the size of its intersection with $e$. The first case is settled.    

If $\ell+2s-a < \ell+s$, we ensure that $g \in \mathcal{F}_s$ by adding $a-s$ elements from the $a$ elements in $e \cap [\ell+2s]$ to it. As in the previous case, we may then add further elements of $e$ to $g$ until their intersection is at most $\ell-1$ and then complete $g$ with elements neither in $[\ell+2s]$ nor in $e$, if needed. To finish the proof, we argue that this extension may be done in more than one way. The first step may be done in $\binom{a}{a-s} \geq 1$ ways. Once these $a-s$ elements are fixed, there are $r-a+s$ elements of $e$ remaining, from which we may still choose up to $\ell-1-a+s$. Clearly, $\ell-1-a+s<r-a+s$, hence the second extension can be done in more than one way unless $\ell-1-a+s=0$, which means that the first step already creates an intersection of size $\ell-1$ between $g$ and $e$. We now suppose the latter. Recall that the first step may be done in $\binom{a}{a-s}$ ways, which is larger than one unless
$a=a-s=\ell-1\geq 1$. However, if this is the case, we have $s=0$ and $|e \cap [\ell]|=\ell-1$, in particular there are $r+1$ elements in $e \cup [\ell]$. In this case, Theorem \ref{otto1} leads to
$$n \geq (\ell+1)(r-\ell+1) \geq 2r -\ell+2 = (r+1) + r-\ell +1,$$
as $(\ell+1)(r-\ell+1)-(2r -\ell+2)=\ell(r-\ell+1) - r - 1\geq 0$
since $2 \leq \ell \leq r-1$, hence $r \geq 3$. This implies that there are at least $r-\ell+1$ elements in $[n]$ outside $[\ell] \cup e$, from which we may easily build $\binom{r-\ell+1}{r-\ell}=r-\ell+1 \geq 2$ elements of $\mathcal{F}_s$ that are not $\ell$-intersecting with $e$. This concludes the proof of the lemma.  
\end{proof}

With Theorem \ref{otto1} and Lemma~\ref{auxlem1}, we are now ready to prove Theorem~\ref{otto2}. 
\begin{proof}[Proof of Theorem~\ref{otto2}]
Let $H=([n],E) $ be an $r$-uniform hypergraph. Consider a maximal $\ell$-intersecting family $F \subseteq E$. Let $\Delta$ be a $(2,\ell)$-coloring of the hyperedges of $H$. For every hyperedge $e \in E \setminus F$ there exists a hyperedge $f \in F$ such that $e$ and $f$ intersect in less than $\ell$ vertices, hence they are colored differently by $\Delta$.
Thus, having fixed the colors of hyperedges in $F$ in any way, the colors of all hyperedges $e \in E$ are uniquely determined. We conclude that
\begin{equation}\label{frida3}
\kappa (H, 2,\ell) \leq 2^{|F|} \leq 2^{ex(n,\mathcal{B}_{r,\ell})}.
\end{equation}
On the other hand, it is easy to see that the number of $(k,\ell)$-Kneser colorings in an $r$-uniform hypergraph whose hyperedges are given by an extremal configuration achieves equality in (\ref{frida3}). Indeed, all the hyperedges are $\ell$-intersecting and may therefore be colored with any of the two colors, independently of the assignment of colors to the other hyperedges.

We now show that such hypergraphs are the only extremal hypergraphs when $n$, $r$ and $\ell$ satisfy the conditions in the statement of Theorem \ref{otto1} and $n>2r$ if $\ell=1$. First, for equality to hold for a hypergraph $H$, the argument above implies
that $H$ contains an $\ell$-intersecting family $F$ of maximum size, which by Theorem \ref{otto1} is a permutation of
$\mathcal{F}_s$ or $\mathcal{F}_{s+1}$. For a contradiction, suppose that $H$ contains an additional
hyperedge $e$ that is not in $F$. By Lemma \ref{auxlem1}, there are at least two elements $f$ and $g$ in $F$ that are not $\ell$-intersecting with $e$. As a consequence, for any $(2,\ell)$-Kneser coloring $\Delta$ of $H$, we must have $\Delta(f)=\Delta(g)$. In particular, $F$ is a maximal $\ell$-intersecting family in $H$ whose members can be colored in at most $2^{|F|-1}$ ways by $(2,\ell)$-Kneser colorings. Thus $H$ is not extremal, concluding the proof. \end{proof}

Note that, in the case $\ell=1$ and $n=2r$, the one-to-one correspondence
between the extremal configurations for $\ex(n,\mathcal{B}_{r,\ell})$ and the extremal hypergraphs with respect to Kneser colorings does not hold.  Indeed, Theorem \ref{otto1} tells us that one of the extremal configurations
in this case would be the family $\mathcal{F}_0$ of all $r$-sets containing the element $1$. Consider the $r$-set $e=\{r+1,\ldots,2r\}$. It is clear that the only element $f \in \mathcal{F}_0$ that does not intersect $e$ is $f=\{1,\ldots,r\}$. In particular, Lemma \ref{auxlem1} does not hold and every Kneser coloring of the $r$-uniform hypergraph $H$ on $[n]$ with hyperedge set $\mathcal{F}_0$ can be extended to a Kneser coloring of the hypergraph $H'$ with the additional hyperedge $e$ by assigning to $e$ the opposite color of $f$, hence $H'$ is also extremal, despite having non-intersecting hyperedges.  
With this observation for $\ell=1$ and $n = 2r$,
 consider a maximal $\ell$-intersecting family ${\mathcal F}$ of
$r$-subsets of $[2r]$ of size $|{\mathcal F}| = \binom{2r-1}{r}$.
 For every $r$-subset $e \subseteq [2r]$, there is a unique $r$-subset $f \subseteq [2r]$ disjoint from it, namely its complement
$\overline{e}$.  Hence, for every family ${\mathcal G}$ of $r$-subsets
of $ [2r]$ with ${\mathcal F} \cap {\mathcal G} = \emptyset$, the union
 ${\mathcal F} \cup {\mathcal G}$ can be $2$-colored by $2^{|{\mathcal F}|}
=2^{\binom{2r-1}{r}}$ colorings. Moreover, this example also shows that,
for any fixed $k \geq 2$, we have
$$
\KC(2r,r,k,1) = k^{\binom{2r-1}{r}} (k-1)^{\binom{2r-1}{r}} = (k(k-1))^{\ex
(2r, {\mathcal B}_{r, \ell})}.
$$

Since Theorem \ref{otto1} gives the extremal configuration $\mathcal{F}_0$ for $n > (\ell +1)(r - \ell +1)$, we deduce from Theorem \ref{otto2}
that for $n > (\ell +1)(r - \ell +1)$ the extremal hypergraph for $KC(n,r,2,\ell)$ is precisely the $(n,r,\ell)$-star $S_{n,r,\ell}$, the $r$-uniform hypergraph on $n$ vertices whose hyperedges are all $r$-subsets of $[n]$ containing a fixed $\ell$-subset.

For Kneser colorings with at least three colors, we
 frequently use the following  lemma.
\begin{lemma}\label{L0}
Let $k \geq 2$ be an integer. All
optimal solutions $s=(s_1,\ldots,s_c)$ to the maximization problem
\begin{equation}\label{eqL1}
\begin{array}{l}
\max \prod_{i=1}^c s_c\\
c, s_1,\ldots,s_c \textrm{ positive integers,}\\
s_1+\cdots + s_c \leq k,
\end{array}
\end{equation}
have the following form.
\begin{itemize}
\item[(a)] If $k \equiv 0 ~(\bmod \, 3)$, then
$c = k/3$ and all the components of $s$ are equal to $3$.

\item[(b)] If $k \equiv 1 ~(\bmod \, 3)$, then either
$c=\left\lceil k/3\right\rceil$,
with exactly two components equal to $2$
and all remaining components equal to $3$, or
$c=\left\lfloor k/3\right\rfloor$,
with exactly one component equal to $4$ and all
remaining components equal to $3$.

\item[(c)] If $k \equiv 2 ~(\bmod \, 3)$, then
$c=\left\lceil k/3\right\rceil$
with exactly
one component equal to $2$ and all remaining components equal to $3$.
\end{itemize}
As a consequence, the optimal value of \eqref{eqL1} is
$\displaystyle{3^{ k/3 }}$ if $k \equiv 0 ~(\bmod \, 3)$, $\displaystyle{4 \cdot 3^{\left\lfloor k/3 \right\rfloor-1}}$ if
$k \equiv 1 ~(\bmod \, 3)$, and $\displaystyle{2 \cdot
3^{\left\lfloor k/3\right\rfloor}}$ if $k \equiv 2 ~(\bmod \, 3)$.
\end{lemma}

\begin{proof} Let $k \geq 2$ be fixed and let $s=(s_1,\ldots,s_c)$ be an optimal solution to (\ref{eqL1}). Note that $s_1+\ldots+s_c=k$, since otherwise
$(s_1+1,s_2,\ldots,s_c)$ would have larger value, contradicting the
optimality of $s$. Moreover, we must have $s_i>1$, for every
$i \in [c]$. Indeed, if one of the components, say $s_c$,
were equal to $1$, the vector $(s_1+1,s_2,\ldots,s_{c-1})$ would have larger value.

If $s$ has a component $ s_j \geq 5$, we can replace it by two
components $3$ and $s-3$ and increase the objective value as $3(s_j-3) >
s_j$. Iterating this we obtain a sequence containing
only the components $2,3,4$. Replacing each $4$ by two components, $2$ and
$2$, does not change the objective value. Finally, we establish that there are
at most two components equal to $2$, since three components equal to $2$
may be replaced by two components equal to $3$ with an increase in
the objective value.

In conclusion, for any $k \geq 2$, an optimal solution has as many $3$'s
as possible such that one can sum
to $k$ with components equal to $2$. This provides
all optimal solutions to the optimization problem (\ref{eqL1})
unless $k \equiv 1~(\bmod \, 3)$, in which case the two occurrences of
 $2$ may be replaced by one occurrence of   $4$
without affecting the objective value.  \end{proof}

\section{Upper bounds on $\KC(n,r,k,\ell)$}\label{sec_upper}

This section is devoted to finding an upper bound on the function $\KC(n,r,k,\ell)$ for any fixed positive integers $r$, $k$ and $\ell$ with $\ell<r$. To do this, we introduce a generalization of the concept of a vertex cover of a graph.
\begin{definition}
For a positive integer $\ell$, an \emph{$\ell$-cover} of a hypergraph $H$ is a set $C$ of $\ell$-subsets of vertices of $H$ such that every hyperedge of $H$ contains an element of $C$. A \emph{minimum $\ell$-cover} of a hypergraph $H$ is an
\emph{$\ell$-cover} of minimum cardinality.
\end{definition}

Note that this definition coincides with the definition of a vertex cover of a
graph or hypergraph $H$ when $\ell=1$. We show that, for $r$, $k$ and $\ell$ fixed, a $(k,\ell)$-colorable hypergraph has a small $\ell$-cover. 
\begin{lemma}\label{small_c}
Let $r$, $k$ and $\ell$ be positive integers with $\ell<r$, and let $H=(V,E)$ be an $r$-uniform $(k,\ell)$-colorable hypergraph. Then $H$ has an 
$\ell$-cover $C$ with cardinality at most $k\binom{r}{\ell}$.
\end{lemma}

\begin{proof}
The $(k,\ell)$-colorability of $H$ ensures that there cannot be more than $k$ hyperedges that pairwise intersect in fewer than $\ell$ vertices. Hence there is a set $S\subseteq E$ of at most $k$ hyperedges such that every hyperedge of $H$ is $\ell$-intersecting with some element of $S$. In particular, the set 
$C=\{t \colon t \subset e \in S, |t|=\ell\}$
is an $\ell$-cover of $H$ with cardinality 
$|C|\leq k\binom{r}{\ell}$. 
\end{proof}

Given the number $k$  of colors, some functions of $k$, which we now define, are frequently used in the remainder of the paper.

\begin{definition}\label{defCND}
Let $k$ be a positive integer. Let $c(k)=\left\lceil k/3 \right\rceil$,
and let the functions $N(k)$ and $D(k)$ be defined by  
\begin{equation*}
\left\{
\begin{array}{llll}
\textrm{if }k \equiv 0 ~(\bmod \, 3),&N(k)=
\frac{k!}{(3!)^{ k/3  }} &\textrm{ and }D(k)=
3^{ k/3 }\\
\textrm{if }k \equiv 1 ~(\bmod \, 3),&N(k)=\binom{\left\lceil k/3
 \right\rceil}{2}\frac{k!}{4 \cdot (3!)^{\left\lceil k/3
\right\rceil-2}
} &\textrm{ and }
 D(k)=4 \cdot 3^{\left\lceil k/3 \right\rceil-2}\\
\textrm{if }k \equiv 2 ~(\bmod \, 3), &N(k)=\left\lceil
k/3 \right\rceil \frac{k!}{2 \cdot (3!)^{\left\lfloor k/3
\right\rfloor} }&\textrm{ and } D(k)=2 \cdot 3^{\left\lfloor k/3
\right\rfloor}.
\end{array}
\right.
\end{equation*}
\end{definition}

\begin{theorem}\label{T1}
Let $r \geq 2$ and $k \geq 3$ be positive integers and fix $\ell \in [r-1]$.
\begin{itemize}
\item[(i)] For $k=3$, there exists $n_0>0$ such that, for every $n \geq n_0$,
\begin{equation}\label{eqT1}
\KC(n,r,3,\ell) \leq 3^{\binom{n-\ell}{r-\ell}}.
\end{equation}
Moreover, for $n \geq n_0$, equality in \eqref{eqT1} is achieved only by the $(n,r,\ell)$-star $S_{n,r,\ell}$, the $r$-uniform hypergraph on $n$ vertices whose hyperedges are all $r$-subsets of $[n]$ containing a fixed $\ell$-set.

\item[(ii)] Given $k \geq 4$, there exists $n_0>0$ such that, for every $n \geq n_0$,
\begin{equation}\label{eqT1_2}
\KC(n,r,k,\ell) \leq N(k)k^{\binom{\ell c(k)}{\ell+1}\binom{n-\ell-1}{r-\ell-1}}D(k)^{\binom{n-\ell}{r-\ell}},
\end{equation}
where $c(k)$, $N(k)$ and $D(k)$ are defined in Definition \ref{defCND}.
\end{itemize}
\end{theorem}
Note that Theorem~\ref{T1}(i) is just Theorem~\ref{thm_2and3}. Moreover, the upper bound on $\KC(n,r,k,\ell)$ given in Theorem~\ref{T1}(ii) is a byproduct of the
considerations proving part (i).  In Sections~\ref{sec_extremal}
and \ref{sec_asy} the asymptotic growth of $\KC(n,r,k,\ell)$ will be determined
precisely; however, these precise expressions are rather involved, as they arise from inclusion-exclusion.

\begin{proof} Let $r$, $k$ and $\ell$ be as in the statement of the theorem, and let $H=(V,E)$ be a $(k,\ell)$-colorable $r$-uniform hypergraph on $n$ vertices.

We start with an overview of the proof, which is structured in terms of a minimum $\ell$-cover  
$C=\{t_1,\ldots,t_c\}$ of $H$. By Lemma~\ref{small_c}, we already know that $c \leq k \binom{r}{\ell}$, so that the size of $C$ may not increase as a function of $n$ when we consider ever larger hypergraphs (with respect to the number of vertices $n$). Let $V_C=\cup_{i=1}^ct_i$ be the set of vertices of $H$ that appear in $C$. The set of hyperedges of $H$ will be split into $E=E' \cup F$, where $e \in E$ is assigned to $E'$ if $|e \cap V_C|=\ell$ and it is assigned to $F$ if $|e \cap V_C|>\ell$.

Since each element of $F$ has intersection at least $\ell+1$ with $V_C$, we have that, for $n$ sufficiently large, the size of $F$ is bounded above by 
$$|F| \leq \binom{\left| V_C
\right|}{\ell+1}\binom{n-|V_C|}{r-\ell-1},$$
which is asymptotically smaller than the largest possible size of $E'$, namely
$$\binom{\left| V_C \right|}{\ell}\binom{n-|V_C|}{r-\ell}.$$
As a consequence, the contribution of the $(k,\ell)$-colorings of $F$ will be treated as an `error', and we will focus on the structure of the colorings of $H'=H \setminus F=H[E']$. 

The main objective here is to show that the largest number of colorings of $H'$ is achieved when the size of the minimum vertex cover $C$ is equal to $c(k)$, and that the number of colorings is exponentially smaller when this is not the case. To this end, we shall show that the bulk of the colorings consists of colorings such that every color appears `many' times and that, when this happens, the coloring must be `star-like', in the sense that, for every given color $\sigma$, there must be a cover element contained in all the hyperedges colored $\sigma$. This will then be used, in conjunction with the proof of Lemma~\ref{L0}, to show that the best way to distribute the colors among the cover elements occurs when $|C|=c(k)$. Once this has been established, it suffices to combine the number of colorings in this setting with the `error' terms to achieve the upper bounds in the statement of the theorem.

We now proceed with a detailed proof of Theorem~\ref{T1}. For each
$\ell$-set $t_i \in C$,
we define the $(r-\ell)$-uniform hypergraph $H_i$ on the
vertex set $V'=V \setminus \bigcup_{i=1}^c
t_i$ such that an $(r-\ell)$-subset $e'$ of $V'$ is a hyperedge
in $H_i$ if and only if $e' \cup t_i$ is a hyperedge of $H$. Let $F$ be the
set of hyperedges of $H$ that do not have an $H_i$ counterpart.
These hyperedges contain at least $\ell+1$ vertices from $\bigcup_{i=1}^ct_i$, so that
$$|F| \leq \binom{\left| \bigcup_{i=1}^ct_i
\right|}{\ell+1}\binom{n-\ell-1}{r-\ell-1}.$$
Let $H'= H \setminus F$ be the subhypergraph of $H$ obtained by removing all
hyperedges in $F$. Moreover, any $(k,\ell)$-coloring of
$H$ is the combination of a $(k,\ell)$-coloring of $H'$ with a coloring of the
hyperedges in $F$ with at most $k$ colors. We know that there are at most
\begin{equation}\label{star1}
\displaystyle{k^{|F|} \leq k^{\binom{\left| \bigcup_{i=1}^ct_i \right|}{\ell+1}\binom{n-\ell-1}{r-\ell-1}} \leq k^{\binom{c\ell}{\ell+1}\binom{n-\ell-1}{r-\ell-1}}}
\end{equation}
colorings of the latter type, thus we now concentrate on $(k,\ell)$-colorings of $H'$.

Consider a $(k,\ell)$-coloring $\Delta$ of $H'$. For each
$\ell$-set $t_i \in C$ and each color
$\sigma \in [k]$, let $H_{i,\sigma}$ be the $(r-\ell)$-uniform
 subhypergraph of $H_i$ induced by the hyperedges of color $\sigma$. We say that $H_{i,\sigma}$ is \textit{substantial} the number of hyperedges in it is larger than
\begin{equation}\label{defL}
L=\max_{0 \leq m \leq \ell-1}{\binom{r-\ell}{\ell-m}\binom{n-2\ell+m}{r-2\ell+m}}.
\end{equation}
Observe that stating that $H_{i,\sigma}$ is substantial formalizes the notion of $\sigma$ appearing `many times', which was mentioned in the outline of the proof. We define $H_i$ to be $s$-\textit{influential} if there are precisely $s$ colors $\sigma$ for which $H_{i,\sigma}$ is substantial.

Our first auxiliary result shows that, if $H_{i,\sigma}$ is substantial, then all hyperedges with color $\sigma$ must contain $t_i$. Hence, given a color $\sigma$, there is at most one value of $i$ such that $H_{i,\sigma}$ is substantial, in which case we say that $\sigma$ is \emph{substantial} for the cover element $t_i$. Intuitively, the subgraph of $H'$ induced by $\sigma$ is a `star' centered at the cover element $t_i$.
\begin{lemma}\label{L1} If the subhypergraph
$H_{i,\sigma}$ is substantial and $e$ is a
hyperedge of $H$ with color $\sigma$, then $t_i \subseteq e$.
In particular, for $i'\neq i$, each
subhypergraph $H_{i',\sigma}$ does not contain any hyperedges.
\end{lemma}

\begin{proof} Suppose for a contradiction that a hyperedge $e \in E$ has color $\sigma$, but $t_i \not \subseteq e$, and let $t_{i'}$ be an element in the $\ell$-cover $C$ contained in $e$.
By definition, the number of hyperedges $h$ in $H_{i,\sigma}$ whose intersection with $e$ has size at least $\ell$ is at most
$$\displaystyle{U=\binom{r-\ell}{\ell-|t_i \cap t_{i'}|}\binom{n-2\ell+
|t_i \cap t_{i'}|}{r-(\ell+|t_i \setminus t_{i'}|)},}$$
since any such $h$ must contain at least $\ell-|t_i \cap t_{i'}|$ elements of
$e\setminus t_{i'}$. Taking the maximum over all possible
sizes $|t_i \cap t_{i'}|$ of the intersection, we have, for $n$ sufficiently large,
$$U \leq \max_{0 \leq m \leq \ell-1}{\binom{r-\ell}{\ell-m}\binom{n-2\ell+m}{r-2\ell+m}} = L.$$
Since the subhypergraph $H_{i,\sigma}$ is substantial,
this is smaller than the number of hyperedges in $H_{i,\sigma}$,
contradicting the fact that the set of hyperedges
in color class $\sigma$ is $\ell$-intersecting. Thus, if $H_{i,\sigma}$ is substantial and $e$ is a hyperedge of $H$ with color $\sigma$, then $t_i$ is indeed a subset of $e$.

To conclude the proof, observe that, for $i'\neq i$, the elements of $H_{i'}$
are determined by all hyperedges $f$ of $H$ whose intersection with
$\bigcup_{m=1}^c t_m$ is equal to $t_{i'}$, hence $f$
does not contain $t_{i}$ and cannot have color $\sigma$. This proves that $H_{i',\sigma}$ has no hyperedges. \end{proof}

An immediate consequence of this lemma is the fact that, if all the colors are substantial for some cover element, then it must hold that, for every cover element $t_i$, there is a color $\sigma$ such that $H_{i,\sigma}$ is substantial.  
\begin{lemma}\label{L2}
If $C=\{t_1,\ldots,t_c\}$ is a minimum $\ell$-cover of $H$
such that there
exists a $(k,\ell)$-coloring $\Delta$ of $H$ for which the subhypergraph
 $H_{i_j}$ is
$s_{i_j}$-influential, where $s_{i_j} \geq 1$
for $j \in [m]$, and $s_{i_1}+\cdots+s_{i_m}=k$, then $m=c$.
\end{lemma}

\begin{proof} Suppose for a contradiction that $m<c$. Since the set $C'=\{t_{i_1},
\ldots,t_{i_m}\}$ is not an $\ell$-cover of $H =(V,E)$,
we may consider a hyperedge
$e \in E$  which does not contain any element from $C'$.
Without loss of generality, assume that $\Delta$ assigns color $k$ to $e$.

However, under the conditions in the statement, Lemma \ref{L1} implies that,
for every color $\sigma \in [k]$, there is $j \in [m]$
such that $H_{i_j,\sigma}$ is substantial. Moreover, all the hyperedges with
color $\sigma$ should contain $t_{i_j}$. This yields a contradiction,
since color $k$ cannot have this property. \end{proof}

We resume the proof of Theorem \ref{T1}. Recall that our objective is to show that the largest number of $(k,\ell)$-colorings is achieved by a hypergraph with $|C|=c(k)$. To this end, we count the colorings of $H'$ according to their distribution of substantial colors: given $j \in \{0,\ldots,k\}$,
let $\mathcal{I}_j$ be the set of all
non-negative integral solutions to the equation $s_1+\cdots +s_c=j$.
For any such vector $s=(s_1,\ldots,s_c)$, let $\Delta_s(H')$ be the
set of all $(k,\ell)$-colorings of $H'$ for which
$H_i$ is $s_i$-influential, for each $i \in [c]$.

An immediate consequence of Lemma \ref{L1} and (\ref{star1}) is
\begin{equation}\label{eq1}
\begin{split}
\kappa(H,k,\ell) &\leq k^{\binom{c\ell}{\ell+1}\binom{n-\ell-1}{r-\ell-1}}\sum_{j=0}^{k} \sum_{s \in \mathcal{I}_j} \left|\Delta_s(H')\right|.
\end{split}
\end{equation}
We now bound the number of colorings in $\Delta_s(H')$ for every fixed
vector $s=(s_1,\ldots,s_c)$ with non-negative integral components such
that $s_1+\cdots +s_c=j$, where $H_i$ is $s_i$-influential for
each $i \in [c]$.
 Clearly, the $j$ colors that contribute for
the hypergraphs $H_1, \ldots, H_c$ to be influential can be chosen in
$\binom{k}{j}$ ways. Moreover, these colors may be distributed
among the hypergraphs $H_1,\ldots, H_c$ in
$\frac{j!}{s_1!s_2!\cdots s_c!}$ ways.
Let $N = \vert \bigcup_{m=1}^c t_m \vert$.   Once the
$j$ colors are
distributed, the hyperedges in $H_i$ may be colored in at most
\begin{eqnarray*}
&& \sum_{(a_1,\ldots,a_{k-j})}
\left(\prod_{t=1}^{k-j}\binom{\binom{n-N}{r-\ell}}{a_t}\right)
s_i^{\binom{n-N}{r-\ell}}
\end{eqnarray*}
ways, if $s_i \geq 1$, where the sum is such that each $a_t$ ranges from $0$
to $L$. This is because $H_i$ contains at most
$\binom{n-N}{r-\ell}$ hyperedges, we may choose $a_t$,
$0 \leq a_t \leq L$,
of them to have each of the $k-j$ colors that do not contribute for an
$H_i$ to be influential, and all the remaining hyperedges may be colored with any of the $s_i$ colors that make $H_i$ $s_i$-influential.
We infer, by using $\binom{n-N}{r-\ell} \geq 2$ and $(x^{\ell+1}-1)/(x-1)
\leq 2x^{\ell}$ for $x \geq 2$, the upper bound
\begin{eqnarray}
&& \sum_{(a_1,\ldots,a_{k-j})}
\left(\prod_{t=1}^{k-j}\binom{\binom{n-N}{r-\ell}}{a_t}\right)
s_i^{\binom{n-N}{r-\ell}}
\leq
\sum_{(a_1,\ldots,a_{k-j})} \left(\prod_{t=1}^{k-j}
\binom{n-N}{r-\ell}^{a_t}\right)
s_i^{\binom{n-N}{r-\ell}} \nonumber \\
&=&
\sum_{(a_1,\ldots,a_{k-j})} \binom{n-N}{r-\ell}^{\sum_{t=1}^{k-j} a_t}
s_i^{\binom{n-N}{r-\ell}}
= \left( \sum_{p=0}^L  \binom{n-N}{r-\ell}^{p} \right)^{k-j}
s_i^{\binom{n-N}{r-\ell}} \nonumber \\
& =&
\left( \frac{\binom{n-N}{r-\ell}^{L+1} - 1}{\binom{n-N}{r-\ell}
-1}\right)^{k-j}
s_i^{\binom{n-N}{r-\ell}}
\leq
  2^{k-j}\binom{n-N}{r-\ell}^{L(k-j)}
s_i^{\binom{n-N}{r-\ell}}. \label{eq89}
\end{eqnarray}

If $s_i=0$, the hypergraph $H_i$ contains at most $L(k-j)$
hyperedges, which may be colored with at most $k-j$ colors in at most
\begin{equation}\label{eqsi0}
\begin{split}
&
(k-j)^{L(k-j)}
\end{split}
\end{equation}
ways.

Let $n_0>0$ be such that, for every $n \geq n_0$, the maximum of $L$ in (\ref{defL}) is attained by $m=\ell-1$, so that
\begin{equation*}
\begin{split}
L&= (r-\ell) \binom{n-\ell-1}{r-\ell-1}=\frac{(r-\ell)^2}{n-\ell} \binom{n-\ell}{r-\ell}.
\end{split}
\end{equation*}
In this case, we may derive the following from (\ref{eq89}) and (\ref{eqsi0}), observing that $c$ is an upper bound on the number of vanishing components in a vector $s=(s_1,\ldots,s_c)$:
\begin{eqnarray}\label{eqUB}
&& \sum_{s \in \mathcal{I}_j} \left| \Delta_s(H') \right| \leq \binom{k}{j}
2^{k-j}\binom{n-N}{r-\ell}^{L(k-j)}(k-j)^{cL(k-j)}\sum_{s \in \mathcal{I}_j}
\frac{j!}{s_1!s_2! \cdots s_c!}\prod_{i=1,s_i \neq 0}^c
s_i^{\binom{n-N}{r-\ell}} \nonumber \\
& =& \binom{k}{j} 2^{(k-j)
+L(k-j) \log{\binom{n-N}{r-\ell}} + cL(k-j) \log{(k-j)}}\sum_{s \in \mathcal{I}_j}
\frac{j!}{s_1!s_2! \cdots s_c!}\prod_{i=1,s_i \neq 0}^c
s_i^{\binom{n-N}{r-\ell}} \nonumber \\
& \leq& \binom{k}{j} 2^{(k-j) + (k-j)(r-\ell + c)\frac{(r-\ell)^2}{n-\ell}
\binom{n-\ell}{r-\ell} \log{n}
}\sum_{s \in \mathcal{I}_j} \frac{j!}{s_1!s_2! \cdots s_c!}\prod_{i=1,s_i
\neq 0}^c s_i^{\binom{n-N}{r-\ell}}.
\end{eqnarray}
In the last step, we used that $k-j \leq n$ and that
$ \log{\binom{n-N}{r-\ell}} \leq
(r-\ell) \log{n}$.

Observe that, for our fixed value of $k$, the product
$\prod_{i=1,s_i \neq 0}^c s_i^{\binom{n-N}{r-\ell}}$
is maximized when the nonzero components of $s$ are the components of a
vector in the set $\mathcal{S}(k)$ of optimal solutions to (\ref{eqL1}),
described in Lemma \ref{L0}. Recall that $D(k)$ given in the statement of
Definition  \ref{defCND} is precisely the optimal value of (\ref{eqL1}),
and, whenever the nonzero components of the integral vector
$s=(s_1,\ldots,s_c)$ are not an optimal solution to (\ref{eqL1}),
 let $\gamma>0$ be such that
\begin{eqnarray} \label{otto3}
&& \prod_{i=1,s_i \neq 0}^c s_i<D(k)^{1-3 \gamma}.
\end{eqnarray}

We are now ready to obtain an upper bound on the number of $(k,\ell)$-colorings of $H'$ associated with solutions of the equation $s_1+\cdots+s_c \leq k$ that are \emph{not} optimal with respect to (\ref{eqL1}). This will be used to show that most of the $(k,\ell)$-colorings of an extremal hypergraph $H'$ must be associated with optimal solutions to~(\ref{eqL1}).  
\begin{lemma}\label{L3}
Let $r \geq 2$, $k \geq 3$ and $\ell$ be positive integers with $\ell < r$. There exists $n_0$ such that, for every $n \geq n_0$,
the following property holds. Let $H$ be an $n$-vertex $r$-uniform hypergraph with an $\ell$-cover $C$ of cardinality $c$ where the union of its elements has size $N$, which
is independent of $n$. Then
$$k^{\binom{c\ell}{\ell+1}\binom{n-\ell-1}{r-\ell-1}}\sum_{j=0}^{k} \sum_{s \in \mathcal{I}_j \setminus \mathcal{S}(k)} \left|\Delta_s(H')\right| \leq D(k)^{\binom{n-N}{r-\ell}(1-\gamma)}.$$
In particular, if $\Delta_s(H')=\emptyset$ for
every $s=(s_1,\ldots,s_c)$ whose nonzero components are the components of a
vector in the set of $\mathcal{S}(k)$ to (\ref{eqL1}), then
$$\kappa(H,r,k) \leq D(k)^{\binom{n-N}{r-\ell}(1-\gamma)}.$$
\end{lemma}

\begin{proof} Let $H$ be such an $r$-uniform hypergraph and choose $n_0$ sufficiently large so that, for every $n \geq n_0$,
\begin{equation*}
\begin{array}{l}
k^{\binom{c\ell}{\ell+1}\binom{n-\ell-1}{r-\ell-1}}=k^{\binom{c\ell}{\ell+1}\frac{r-\ell}{n-r}\binom{n-\ell-1}{r-\ell}}<D(k)^{\gamma\binom{n-N}{r-\ell}} \textrm{ and }\\
(k+1)! \binom{k+c-1}{c-1}
2^{2k + k(r-\ell +c)\frac{(r-\ell)^2}{n-\ell} \binom{n-\ell}{r-\ell} \log{n}
}<D(k)^{\gamma\binom{n-N}{r-\ell}}.
\end{array}
\end{equation*}

The inequalities  (\ref{eqUB}) and (\ref{otto3}) imply that
\begin{eqnarray}\label{eq2b}
&& k^{\binom{c\ell}{\ell+1}\binom{n-\ell-1}{r-\ell-1}}\sum_{j=0}^{k} \sum_{s \in \mathcal{I}_j \setminus \mathcal{S}(k)} \left|\Delta_s(H')\right| \nonumber \\
& \leq&  D(k)^{\gamma\binom{n-N}{r-\ell}} \sum_{j=0}^{k} \binom{k}{j}
2^{(k-j) + (k-j)(r-\ell +c)\frac{(r-\ell)^2}{n-\ell} \binom{n-\ell}{r-\ell}
\log{n}} \binom{j+c-1}{c-1} j! D(k)^{\binom{n-N}{r-\ell}(1-3 \gamma)}
\nonumber \\
& \leq& (k+1)! \binom{k+c-1}{c-1}2^{2k + k(r-\ell
+c)\frac{(r-\ell)^2}{n-\ell} \binom{n-\ell}{r-\ell} \log{n}
}
D(k)^{\binom{n-N}{r-\ell}(1-2 \gamma)} \nonumber \\
&\leq&  D(k)^{\binom{n-N}{r-\ell}(1-\gamma)},
\end{eqnarray}
as required. Here, we are using the facts that $\mathcal{I}_j$, the set of
non-negative integral solutions to the equation $s_1+\cdots+s_c=j$,
has size $\binom{j+c-1}{c-1}$,
and that $\binom{k}{j}2^{k-j} \leq 2^{2k}$. The term $(k+1)!$ comes from
the multiplication of $k!$, the maximum possible value attained by $j!$ in the sum, by the number $k+1$ of summands.

When $\Delta_s(H')=\emptyset$ for
every $s=(s_1,\ldots,s_c)$, whose nonzero components are the components of a
vector in $\mathcal{S}(k)$, the fact that $\displaystyle{\kappa(H,r,k) \leq D(k)^{\binom{n-N}{r-\ell}(1-\gamma)}}$ is an immediate consequence of inequality
(\ref{eq1}) and the above. \end{proof}

To conclude the proof, we use the above discussion to prove the validity of (i) and (ii) in
Theorem \ref{T1}. For part (i), let $H =(V,E)$ with $|V| = n$
be a $(3,\ell)$-colorable $r$-uniform hypergraph and let
$C=\{t_1,\ldots,t_c\}$ be a minimum $\ell$-cover of $H$.

If $c=1$, we may use the immediate bound
$$\kappa(H,3,\ell) \leq 3^{|E|} \leq 3^{\binom{n-\ell}{r-\ell}},$$
with equality occurring if and only if $H$ is isomorphic to $S_{n,r,\ell}$, the $r$-uniform hypergraph on $n$ vertices whose hyperedges are all $r$-subsets of $[n]$ containing a fixed $\ell$-subset.

Now, suppose that $c>1$. By Lemma~\ref{small_c}, the $(3,\ell)$-colorability of $H$ ensures that
$c \leq 3 \binom{r}{\ell}$,
so that $c$ is independent of $n$. Moreover, with $c>1$, Lemma \ref{L2}
implies that $|\Delta_{s}(H')|=0$ for every vector $s=(s_1,\ldots,s_c)$ for which one of the entries is equal to $3$.
As a consequence, Lemma \ref{L3} with $r+1 \leq N \leq 3r$, where $\displaystyle{N=\left| \cup_{i=1}^{c}t_i \right|}$, implies that
\begin{equation*}
\kappa(H,3,\ell) \leq D(3)^{\binom{n-N}{r-\ell}(1-\gamma)} < 3^{\binom{n-\ell}{r-\ell}}
\end{equation*}
for $n$ sufficiently large. This proves part (i) of Theorem~\ref{T1}.

We now establish part (ii). First, we consider the  simpler case
$\displaystyle{k \not \equiv 1 ~(\bmod \, 3)}$. Fix a $(k,\ell)$-colorable
$r$-uniform hypergraph $H$ on $n$ vertices. Again, we choose a minimum $\ell$-cover
$C=\{t_1,\ldots,t_{c}\}$ of $H$. Observe that $c \leq k \binom{r}{\ell}$ is independent of $n$ by Lemma~\ref{small_c}.

Recall that $\mathcal{S}(k)$ is the set of optimal solutions
$s=(s_1,\ldots,s_c)$ of the maximization problem (\ref{eqL1}) given by Lemma
\ref{L0}. By Lemmas \ref{L3} and  \ref{L2}, if $c \neq c(k)$, we have
\begin{equation}\label{tapsi2}
\kappa(H,k,\ell) \leq D(k)^{\binom{n-N}{r-\ell}(1-\gamma)}.
\end{equation}
If  $c=c(k)$,
inequality  (\ref{eq1}) leads to
\begin{eqnarray}\label{tapsi3}
&& \kappa(H,k,\ell) \nonumber \\
&\leq& k^{\binom{c(k)\ell}{\ell+1}\binom{n-\ell-1}{r-\ell-1}}\sum_{j=0}^{k} \sum_{s \in \mathcal{I}_j} \left|\Delta_s(H')\right| \nonumber \\
&=& k^{\binom{c(k)\ell}{\ell+1}\binom{n-\ell-1}{r-\ell-1}}
\left(\sum_{s \in \mathcal{S}(k)} \left|\Delta_s(H')\right| +
\sum_{j=0}^{k-1} \sum_{s \in \mathcal{I}_j} \left|\Delta_s(H')\right| +
\sum_{s \in \mathcal{I}_k \setminus \mathcal{S}(k)} \left|\Delta_s(H')\right|
\right).
\end{eqnarray}
On the one hand, using (\ref{eqUB}) with $j=k$, and
 with the sum restricted to $\mathcal{S}(k)$, we obtain
\begin{eqnarray}\label{tapsi4}
\sum_{s \in \mathcal{S}(k)} \left|\Delta_s(H')\right| &\leq&
\sum_{s \in \mathcal{S}(k)} \frac{k!}{s_1!s_2! \cdots s_c!}\prod_{i=1,s_i
\neq 0}^c s_i^{\binom{n-N}{r-\ell}}
= N(k) D(k)^{\binom{n-N}{r-\ell}}.
\end{eqnarray}
Note that $N(k)$ is precisely the number of optimal solutions of (\ref{eqL1})
multiplied by the coefficient $\frac{k!}{s_1!s_2! \cdots s_c!}$. This product is
the same for every $s \in \mathcal{S}(k)$, as
$\displaystyle{k \not \equiv 1 ~(\bmod \, 3)}$.

On the other hand, with calculations as in  (\ref{eq2b}), we derive
\begin{equation}\label{tapsi5}
\begin{split}
\sum_{j=0}^{k-1} \sum_{s \in \mathcal{I}_j} \left|\Delta_s(H')\right| +
\sum_{s \in \mathcal{I}_k \setminus \mathcal{S}(k)} \left|\Delta_s(H')\right|
\leq D(k)^{(1- \gamma)\binom{n-N}{r-\ell}},
\end{split}
\end{equation}
which, for any fixed $\eps > 0$,
is smaller than $\displaystyle{\eps  D(k)^{\binom{n-N}{r-\ell}} }$ if $n$ is sufficiently large.

Note that $D(k)^{\binom{n-\ell}{r - \ell}} \geq 2
D(k)^{\binom{n - N}{r - \ell}}$
for $k \geq 4$, since $c(k) \geq 2$ implies $N \geq \ell +1$ in this case.
As a consequence of (\ref{tapsi4}) and (\ref{tapsi5}), we have
\begin{equation*}
\KC(n,r,k,\ell)=\max_H\kappa(H,k,\ell)\leq k^{\binom{c(k)\ell}{\ell+1}\binom{n-\ell-1}{r-\ell-1}}N(k)D(k)^{\binom{n-\ell}{r-\ell}},
\end{equation*}
which implies the upper bound in the statement of Theorem \ref{T1}(ii).

If $k \equiv 1 ~(\bmod \, 3)$, the proof requires some additional work.
Recall from Lemma \ref{L0} that, in this case, there are two essentially
different optimal solutions to  (\ref{eqL1}),
each containing $(\left\lceil k/3 \right\rceil-2)$ many
$3$'s, but one containing
two 2's, while the other contains one $4$. We now mimic the proof of the
previous case, omitting some of the details. As in
(\ref{tapsi2}), we obtain
$$\kappa(H,k,\ell) \leq D(k)^{\binom{n-N}{r-\ell}(1-\gamma)}$$
whenever $H$ is an $r$-uniform $n$-vertex hypergraph with minimum $\ell$-cover
of size $c \notin \{c(k)-1,c(k)\}$, where
$c(k)=\left\lceil k/3 \right\rceil$, since these are the two cover sizes corresponding to optimal solutions of (\ref{eqL1}). If $c=c(k)-1$,
we repeat the arguments used in (\ref{tapsi3})--(\ref{tapsi5}),
with $\mathcal{S}(k)$ being replaced by the set of
optimal solutions of (\ref{eqL1}) containing one $4$
and $N(k)$ being replaced by $$N'(k)=\left\lfloor \frac{k}{3}
 \right\rfloor \frac{k!}{4! (3!)^{\left\lfloor k/3 \right\rfloor-1}}.$$
The latter is just the number of ways of partitioning the $k$ available
colors among the sets in the cover of size $c(k)-1$ in such a way that the
sizes of the sets in the partition give an optimal solution of (\ref{eqL1})
containing one $4$. This leads to the upper bound
$$\kappa(H,k,\ell)\leq k^{\binom{(c(k)-1)\ell}{\ell+1}
\binom{n-\ell-1}{r-\ell-1}}N'(k)D(k)^{\binom{n-\ell}{r-\ell}}.$$
When $c=c(k)$, we may obtain the following bound using the same arguments:
$$\kappa(H,k,\ell)\leq k^{\binom{c(k)\ell}{\ell+1}
\binom{n-\ell-1}{r-\ell-1}}N(k)D(k)^{\binom{n-\ell}{r-\ell}}.$$
Since $$N'(k)=\left\lfloor \frac{k}{3} \right\rfloor \frac{k!}{4! \cdot
(3!)^{\left\lfloor
k/3 \right\rfloor-1}}<\binom{\left\lceil k/3 \right\rceil
}{2} \frac{k!}{2 \cdot 2 \cdot (3!)^{\left\lfloor k/3 \right\rfloor-1}
}=N(k),$$
we deduce with $N \geq \ell +1$ in both cases that
\begin{equation*}
\KC(n,r,k,\ell) \leq k^{\binom{c(k)\ell}{\ell+1}\binom{n-\ell-1}{r-\ell-1}}N(k)D(k)^{\binom{n-\ell}{r-\ell}},
\end{equation*}
which finishes the proof of Theorem \ref{T1}. \end{proof}

\section{Extremal hypergraphs}\label{sec_extremal}

Theorem \ref{otto2} and part (i) of Theorem \ref{T1} give, for $k= 2$ and
$k= 3$, the exact value of $\KC(n,r,k,\ell)$ for sufficiently large $n$.
Moreover, they determine that the set of extremal $r$-uniform hypergraphs
 $H$, i.e., of hypergraphs with the maximum number of $(k,\ell)$-Kneser
colorings, is precisely the set of $(n,r,\ell)$-stars, the $r$-uniform hypergraphs on $n$ vertices whose hyperedges are all $r$-subsets of $[n]$ containing a fixed $\ell$-subset. In this section, we find properties of extremal hypergraphs for larger values of $k$. In some cases, these properties determine precisely the set of extremal hypergraphs, in others, they only characterize families containing all the extremal hypergraphs. However, the number of $(k,\ell)$-Kneser colorings of all the hypergraphs in these families are ``almost" extremal, in a sense to be made precise in Section~\ref{sec_asy}.

The following two results, one concerning $k=4$ and the other $k \geq 5$, give properties of the hypergraphs with most $(k,\ell)$-Kneser colorings when $n$ is sufficiently large. The proof of the first theorem relies heavily on the arguments used to demonstrate the second theorem and therefore is postponed to the end of this section.

\begin{theorem}\label{T2_4}
Let $r \geq 2$ and $1 \leq \ell < r$ be positive integers. Given a positive integer $n$, let $H^{\ast}$
be an $r$-uniform hypergraph on $n$ vertices satisfying
$$\kappa(H^{\ast},4,\ell)=\KC(n,r,4,\ell).$$
There exists $n_0>0$ such that, if $n \geq n_0$, then $H^{\ast}$ is
isomorphic to $H_{C,r}(n)$ for some set $C=\{t_1,t_2\}$, where both
$t_1$ and $t_2$
are distinct $\ell$-subsets of $[n]$.
\end{theorem}

\begin{theorem}\label{T2}
Let $r \geq 2$, $k \geq 5$ and $1 \leq \ell < r$ be positive integers.
Given a positive integer $n$, let $H^{\ast}$
be an $r$-uniform hypergraph on $n$ vertices satisfying
$$\kappa(H^{\ast},k,\ell)=\KC(n,r,k,\ell).$$
Then there exists $n_0>0$ such that, if $n \geq n_0$, the following holds.
\begin{itemize}
\item[(a)] If $r \geq 2 \ell-1$, then $H^{\ast}$ is isomorphic to $H_{n,r,k,\ell}$.

\item[(b)] If $r < 2 \ell-1$, then $H^{\ast}$ is isomorphic to
$H_{C,r}(n)$ for some set $C=\{t_1,\ldots,t_{c(k)}\}$, where $t_1,
\ldots , t_{c(k)}$ are $\ell$-subsets
of $[n]$ such that $|t_i \cup t_j|>r$, for every $i,j \in [c(k)]$, $i \neq j$.
\end{itemize}
\end{theorem}

\begin{proof} Fix positive integers
$k \geq 5$ and $r \geq 2$,  and $1 \leq \ell < r$. Given a positive
integer $n$, let $H^{\ast}$ be an $r$-uniform $n$-vertex
hypergraph  which satisfies
$\kappa(H^{\ast},k,\ell)=\KC(n,r,k,\ell)$.

As in the proof of Theorem \ref{T1},
we first focus on the case $k \not \equiv 1~(\bmod \, 3)$. We then adapt the proof for the case $k \equiv 1~(\bmod \, 3)$.
Let $D(k)$ be the optimal value and $\mathcal{S}(k)$ be the set of all optimal solutions $s=(s_1, \dots , s_{c(k)})$ of~(\ref{eqL1}), where $c(k)=\lceil k/3\rceil$ is also the number of components of such an optimal solution. We start with the following lower bound, which is easy to derive and very useful in upcoming considerations.
\begin{lemma} \label{lemma123new}
\begin{equation}\label{tapsi6}
\KC(n,r,k, \ell) \geq  \kappa(H_{n,r,k,\ell},k,\ell) \geq \left(\prod_{i=1}^{c(k)}s_i\right)^{\binom{n-\ell c(k)}{r-\ell}} = D(k)^{\binom{n-\ell c(k)}{r-\ell}}.
\end{equation}
\end{lemma}
\begin{proof}
Let $t_1,\ldots,t_{c(k)}$ be the mutually disjoint sets
in the $\ell$-cover $C$ of $H_{n,r,k,\ell}$,
let $s=(s_1,\ldots,s_{c(k)}) \in \mathcal{S}(k)$ and consider a
partition of the set of $k$ colors into sets $S_i$ with $|S_i|=s_i$,
$i \in [c(k)]$. Now, a $(k,\ell)$-Kneser coloring of $H_{n,r,k,\ell}$
can be obtained from any assignment of colors in $S_1$ to hyperedges containing $t_1$ as a subset, and any assignment of colors in $S_i$ to hyperedges with $\ell$-subset $t_i$, but not containing an $\ell$-set in $\{t_1,\ldots,t_{i-1}\}$, $i=2,\ldots,c(k)$. Thus, for the number of
$(k,\ell)$-colorings we obtain
\begin{eqnarray*} \label{eq300}
\KC(n,r,k,l) \geq  \kappa(H_{n,r,k,\ell},k,\ell)
 &\geq&
s_1^{\binom{n-\ell}{r-\ell}}s_2^{\binom{n-2\ell}{r-\ell}}\cdots s_{c(k)}^{\binom{n-c(k)\ell}{r-\ell}} \geq D(k)^{\binom{n-c(k)\ell}{r-\ell}}.
\end{eqnarray*}
\end{proof}

As in (\ref{otto3}), let $\gamma>0$ be such that
$$\displaystyle{\prod_{i=1,s_i \neq 0}^{c(k)} s_i<D(k)^{1-3 \gamma}}$$
whenever the nonzero components of the vector $s=(s_1,\ldots,s_{c})$ are
not an optimal solution to (\ref{eqL1}). Lemmas \ref{L3} and \ref{L2} imply that, for $n$ sufficiently large, if $H$
is an $r$-uniform hypergraph on $n$ vertices with minimum $\ell$-cover
of size $c$ satisfying
$$\kappa(H,k,\ell) \geq D(k)^{\binom{n-c(k)\ell}{r-\ell}(1-\gamma/2)},$$
then $c=c(k)$. Thus, equation (\ref{tapsi6}) implies that
a minimum $\ell$-cover of $H^{\ast}$ has size $c(k)$.

For later reference, we state the following fact as a remark.
\begin{remark}\label{rem_cover}
Fix positive integers $r>\ell$ and $k \not \equiv 1 ~(\bmod \, 3)$. Then, for every $\delta>0$, there exists $n_0>0$ such that any $r$-uniform hypergraph $H$ on $[n]$, $n>n_0$, with minimum cover of size $c \neq c(k)$ satisfies $\kappa(H,k,\ell)<\delta \KC(n,r,k,\ell)$.
\end{remark}

The remainder of the proof has two main parts. First, we  establish
that an extremal hypergraph $H^{\ast}$ for our property must be complete,
that is, it must contain every hyperedge that contains some
$\ell$-set in $C$. With this in hand, we
then prove (a) and (b) in Theorem~\ref{T2}
 by analyzing the interplay between overlappings
in the cover and the number of $(k,\ell)$-colorings.

\begin{lemma} \label{124lemmmanew}
Let $k,r \geq 2$,  and $1 \leq \ell < r$. Let $H^{\ast}=(V,E)$
be an $r$-uniform $n$-vertex
hypergraph with minimum $\ell$-cover $C$  which satisfies
$$\kappa(H^{\ast},k,\ell)=\KC(n,r,k,\ell).$$
Then there exists $n_0$, such that for every integer $n \geq n_0$
the hypergraph $H^{\ast}$ is complete, i.e., every $r$-subset of $V$
containing some set $t \in C$ is a hyperedge of $H^{\ast}$.
\end{lemma}

\begin{proof}
Let $H$ be an $r$-uniform hypergraph with minimum $\ell$-cover $C=\{t_1,
\ldots,t_c\}$, $c = c(k)$, and assume that $H$ is not complete.
Let $U=\bigcup_{i=1}^c t_i$ and $N=|U|$. Consider the case when there is an element $t_i$ in $C$
that covers at most $kL$ hyperedges
not covered by any other element of $C$, where, for $n$ sufficiently large,
$$L=(r-\ell) \binom{n-\ell-1}{r-\ell-1}$$
is precisely the quantity defined in (\ref{defL}).
Let $E_i$ be this set of
hyperedges, which is nonempty, since $C$ is a minimum $\ell$-cover.
Consider the $r$-uniform subhypergraph $H'=H \setminus E_i$ obtained from $H$ by
removing all hyperedges in $E_i$. Let $\mathcal{C}$ be the set of $(k,\ell)$-colorings of $H$ and let $\mathcal{C}_i$ and $\mathcal{C}'$ be the sets obtained by restricting the colorings in $\mathcal{C}$ to $E_i$ and $H'$, respectively. Clearly,
$$\kappa(H,k,\ell)=|\mathcal{C}|\leq |\mathcal{C}_i||\mathcal{C}'| \leq k^{kL} |\mathcal{C}'|.$$
On the other hand, given a coloring $\Delta \in \mathcal{C}$, there is a color $\sigma_i$ assigned by $\Delta$ to an element of $E_i$, since the latter is nonempty. In particular, Lemma \ref{L1} implies that $H_{j,\sigma_i}$ cannot be substantial for $j \neq i$ and, in particular, the restriction of $\Delta$ to $H'$ may have at most $k-1$ colors $\sigma$ for which $H'_{j,\sigma}=H_{j,\sigma}$ is
substantial. Hence, by Lemma~\ref{L3}, if $n$ is sufficiently large,
$$|\mathcal{C}'| < D(k)^{\binom{n-N}{r-\ell}(1-\gamma)}.$$
Moreover, for $n$ sufficiently large, we also have
$$k^{kL}=k^{k(r-\ell) \binom{n-\ell-1}{r-\ell-1}}<D(k)^{\gamma/2
\binom{n-N}{r-\ell}},$$
so that $$\kappa(H,k,\ell) < D(k)^{(1 - \gamma/2)\binom{n-N}{r-\ell}},$$
thus, with~(\ref{tapsi6}) the hypergraph  $H$ is not extremal for the property
of having the largest number of $(k, \ell)$-colorings.

Now, assume that every element in $C$ covers more than
$\displaystyle{kL}$ hyperedges not covered by any other
element of $C$. Let $e$ be an $r$-subset of $V$ containing $t_i \in C$ that
is not a hyperedge of $H$, and define $E_i$  as before. Such an $e$ exists
by the assumption that the hypergraph $H$ is not complete.  Let $\Delta$ be a $(k,\ell)$-Kneser coloring of $H$. By the pigeonhole principle, at least one of the colors, say $\sigma$, appears more than
$L$ times in $E_i$. Moreover, with counting arguments as in the proof of Lemma \ref{L1}, it is easy to see that, if a hyperedge $f$ were in this color class but did not contain $t_i$, then the number of elements of $E_i$ that share an $\ell$-subset with $f$ would be at most $L$, a contradiction. Hence all the
hyperedges assigned color $\sigma$ by $\Delta$ must contain $t_i$, so that $\Delta$ may be
extended to a Kneser-coloring of $H \cup \{e\}$ by assigning color
$\sigma$ to $e$.

Furthermore, there is at least one
$(k,\ell)$-Kneser coloring of $H$ using exactly $c$ colors,
namely the one that assigns color $1$ to all hyperedges containing $t_1$ and
color $i$ to all hyperedges containing $t_i$, but not containing an $\ell$-subset in the set
$\{t_1,\ldots,t_{i-1}\}$ for $i = 2, \ldots, c$. Since
$c=c(k) = \left\lceil k/3 \right\rceil
 \leq k-1$ for $k \geq 3$, we have at least two
options to color $e$, one using a color already used, and one using a new
color. As a consequence, the hypergraph $H \cup \{e\}$ has more $(k,\ell)$-colorings than $H$, establishing that such an $r$-uniform hypergraph $H$ cannot be extremal for the property of having the largest number of $(k,\ell)$-colorings.
\end{proof}

By Lemma~\ref{124lemmmanew}, we may assume in the following that $H$ is complete. Observe that the same conclusion could be reached when $k \equiv 1~(\bmod \, 3)$, but, unlike in the previous case, the size $c$ of a minimum $\ell$-cover of $H$ cannot be uniquely determined, we only know that $c \in \{c(k)-1,c(k)\}$.

In the notation of Definition~\ref{def_extremal}, let $H=H_{C,r}(n)$ for some
set $C=\{t_1,\ldots,t_c\}$ of $\ell$-subsets of $[n]$, $c=c(k)$
(if $k \equiv 1~(\bmod \, 3)$, we need to consider the case $c=c(k)-1$ as well).
It is clear that,
for $n$ sufficiently large, the set $C$ is the unique $\ell$-cover with minimum size of $H$.

To prove Theorem~\ref{T2}, we introduce a special class of
$(k,\ell)$-colorings of $H$.

\begin{definition}
 Let $s=(s_1,\ldots,s_c)$ be an optimal solution to \eqref{eqL1} and let $P(s)=(P_1,\ldots,P_c)$ be an ordered partition of set $[k]$ of colors into
sets such that $|P_i|=s_i$, for every $i \in [c]$. The set of
\emph{$(s,P(s))$-star colorings of $H$}, denoted by $\SC(H,P(s),k,\ell)$,
consists of all $(k,\ell)$-colorings $\Delta$ of $H$ such that, if $\sigma$
lies in $P_i$ and $\Delta(e)=\sigma$, then $e \supset t_i$. The set of
\emph{star colorings of $H$} is defined as
$$\SC(H,k,\ell)=\bigcup_{s \in \mathcal{S}(k)} \bigcup_{P \in \mathcal{P}_s} \SC(H,P,k,\ell),$$
where $\mathcal{P}_s$ is the set of all ordered partitions
$P(s)=(P_1,\ldots,P_c)$ of the set $[k]$ of colors such that $|P_i|=s_i$,
for every $i \in [c]$.

Moreover, set
$$\s(H,k,\ell)=|\SC(H,k,\ell)|.$$
\end{definition}

In other words, the set of star colorings of $H$ is the set of all colorings obtained by first splitting the set of $k$ available colors amongst the cover elements, so that the number of colors assigned to each cover element is given by an optimal solution to~(\ref{eqL1}), and then assigning to each hyperedge a color associated with a cover element contained in it.

The relevance of star colorings is highlighted by the following two results. The first uses the fact that star colorings generalize the special class of colorings considered in Section \ref{sec_upper} to establish that  the set of star colorings of a hypergraph $H$ provides a good approximation of the set of all $(k,\ell)$-colorings. The second result introduces a formula to approximate $\s(H,k,\ell)$, which, for extremal hypergraphs on $n$ vertices, gives the correct asymptotic value of $\KC(n,r,k,\ell)$.
Recall that $U=\bigcup_{i=1}^ct_i$ with
$N= \vert U \vert$.

\begin{lemma}\label{L5}
Let $k, \ell, r$ be positive integers with $\ell < r$. Then there exists $n_0$ such that, for every
$r$-uniform hypergraph $H$ on $n \geq n_0$
vertices with $\ell$-cover of size $c = c(k)$, we have
$$\kappa(H,k,\ell)-D(k)^{\binom{n-N}{r-\ell}(1-\gamma)} \leq \s(H,k,\ell) \leq \kappa(H,k,\ell).$$
\end{lemma}

\begin{proof} It is clear from the definition that $\SC(H,k,\ell)$ is contained in the set of all $(k,\ell)$-colorings of $H$, hence $\s(H,k,\ell) \leq \kappa(H,k,\ell)$. Split the set of all $(k,\ell)$-colorings of $H$ into the set $\mathcal{S}$ of star colorings and the set $\bar{\mathcal{S}}$ of remaining colorings. As a consequence, the result follows if we show that
the number of colorings in $\bar{\mathcal{S}}$ is
at most $\displaystyle{D(k)^{\binom{n-N}{r-\ell}(1-\gamma)}}$.

With the terminology of the proof of Theorem~\ref{T1}, let $\mathcal{C}$ be the family of $(k,\ell)$-Kneser colorings of $H$ for which:
\begin{itemize}
\item[(i)] every color $\sigma$ is such that $H_{i,\sigma}$ is substantial for
some $i \in [c]$;

\item[(ii)] the vector $(s_1,\ldots,s_c)$ lies in $\mathcal{S}(k)$, where
 $s_i=|\{\sigma :   H_{i,\sigma} \textrm{ is substantial}\}|$  
for each $i \in [c]$.
\end{itemize}
Combining inequality (\ref{eq1}) and Lemma \ref{L3}, we see, for $n$
sufficiently large, that the
set $\bar{\mathcal{C}}$ of all $(k,\ell)$-Kneser colorings of $H$ that are not in $\mathcal{C}$ satisfies
$$|\bar{\mathcal{C}}| \leq k^{\binom{c\ell}{\ell+1}\binom{n-\ell-1}{r-\ell-1}}\sum_{j=0}^{k} \sum_{s \in \mathcal{I}_j \setminus \mathcal{S}(k)} \left|\Delta_s(H')\right| \leq D(k)^{\binom{n-N}{r-\ell}(1-\gamma)}.$$
On the other hand, Lemma \ref{L2} tells us that, if $\Delta$ is a
$(k,\ell)$-Kneser coloring and $\sigma$ is a color for which $H_{i,\sigma}$ is
substantial with respect to some set $t_i$ in the cover, then $\Delta$ may
only assign color $\sigma$ to hyperedges containing $t_i$. Hence, any
coloring in $\mathcal{C}$ is also a star coloring, i.e., ${\mathcal C}
\subseteq {\mathcal S}$, thus  $|\bar{\mathcal{S}}|
\leq |\bar{\mathcal{C}}| \leq D(k)^{\binom{n-N}{r-\ell}(1-\gamma)}$
for $n$
sufficiently large,
 as required. \end{proof}

\begin{lemma}\label{L6}
Let $k, \ell, r$ be positive integers with $\ell < r$.
There exists $n_0>0$ such that, for every $n \geq n_0$,
the $(C,r)$-complete hypergraph $H=H_{C,r}(n)$ satisfies
$$\left(1-A \left(\frac{k-1}{k}\right)^{\binom{n-2\ell}{r-\ell}}\right) \sum_{s \in \mathcal{S}(k)} \sum_{P \in \mathcal{P}_s}  \prod_{e \in E} \left( \sum_{t_i \subset e} s_i \right) \leq  \s(H,k,\ell) \leq  \sum_{s \in \mathcal{S}(k)} \sum_{P \in \mathcal{P}_s} \prod_{e \in E} \left( \sum_{t_i \subset e} s_i \right),$$
where $A=A(k)$ is a function of $k$.
\end{lemma}

\begin{proof} The upper bound on $\s(H,k,\ell)$ follows directly from the definition. Indeed, given $s \in \mathcal{S}(k)$ and $P \in \mathcal{P}_s$ each hyperedge $e \in E$ may be assigned $\sum_{t_i \subset e}s_i$ colors by colorings in $\SC(H,P,k,\ell)$. Conversely, any such assignment gives a different $(s,P)$-star coloring, so that
$$|\SC(H,P,k,\ell)|=\prod_{e \in E} \left( \sum_{t_i \subset e} s_i \right),$$
and, as a consequence,
$$\s(H,k,\ell)=|\SC(H,k,\ell)|\leq \sum_{s \in \mathcal{S}(k)} \sum_{P \in \mathcal{P}_s} \prod_{e \in E} \left( \sum_{t_i \subset e} s_i \right).$$

To find a lower bound on $\s(H,k,\ell)$, we
bound from above
 the number of colorings that appear in multiple terms of the union
$$\bigcup_{s \in \mathcal{S}(k)} \bigcup_{P \in \mathcal{P}_s}
\SC(H,P,k,\ell).$$
Let $\Delta$ be a $(k,\ell)$-coloring $\Delta$ lying in $\SC(H,P,k,\ell) \cap
\SC(H,P',k,\ell)$, where $P=(P_1,\ldots,P_c) \in \mathcal{P}_{s}$ and $P'=(P'_1,\ldots,P'_c) \in \mathcal{P}_{s'}$, for $s,s' \in \mathcal{S}(k)$
not necessarily distinct. Then,
there must be a color $\sigma \in [k]$ such that
$\sigma \in P_i \cap P'_j$, $i \neq j$, so that every hyperedge assigned color
$\sigma$ by $\Delta$ contains both $t_i$ and $t_j$.

For $s,s'$, $P,P'$, $i$ and $j$ fixed, the number of colorings
with the above property is at most $M(s,s',P,P',i,j)$ with
\begin{equation*}
\begin{split}
M(s,s',P,P',i,j)& \leq \min\{s_i,s_j'\} \prod_{e \in E \setminus S_{i,j}} \left( \sum_{t_m \subset e} s_m  \right)\prod_{e \in S_{i,j}} \left( \sum_{t_m \subset e} s_m-1 \right)\\
& \leq  4 \left(\prod_{e \in E} \left( \sum_{t_m \subset e} s_m  \right)\right) \left(\prod_{e \in S_{i,j}}   \frac{
\sum_{t_m \subset e} s_m-1}{\sum_{t_m \subset e} s_m} \right),
\end{split}
\end{equation*}
where $S_{i,j}$ is the set of hyperedges of $H$ that contain $t_i$ but do
not contain $t_j$. Here, the term $\min\{s_i,s_j'\}$ is an upper bound on
the number of possible choices for the color $\sigma$.
The description of the set $\mathcal{S}(k)$ in Lemma \ref{L0} tells us that
$\min\{s_i,s_j'\} \leq 4$.
Finally, the term $\prod_{e \in S_{i,j}}
\left( \sum_{t_m \subset e} s_m-1 \right)$
accounts for the fact that $\sigma$ cannot be used to color the hyperedges
in $S_{i,j}$.

Since the hypergraph $H$ is complete,
$$|S_{i,j}| \geq \binom{n-|t_i \cup t_j|}{r-\ell} \geq \binom{n-2\ell}{r-\ell}.$$
Moreover, since
$\sum_{t_m \subset e} s_m \leq k$,
we have
$$
\frac{\sum_{t_m \subset e} s_m - 1}{\sum_{t_m
\subset e} s_m} \leq \frac{k-1}{k}.
$$
As a consequence,
$$M(s,s',P,P',i,j) \leq 4 \left(\frac{k-1}{k}\right)^{\binom{n-2\ell}{r-\ell}} \prod_{e \in E} \left( \sum_{ t_m \subset e} s_i \right).$$
Now, a generous upper bound on the number of possibilities for $s,s'$, $P,P'$, $i$ and $j$ is
$$\binom{|\mathcal{S}(k)|+1}{2} \binom{k!}{2} \binom{c}{2}k,$$
since there are at most $\binom{|\mathcal{S}(k)|+1}{2}$ ways
of choosing one or two elements of $\mathcal{S}(k)$, there are at most $k!$
partitions of the set $[k]$ of colors
into sets $P_1,\ldots,P_c$, and to choose two of them,
there are at most $\binom{c}{2}k$ ways of choosing two
 elements $t_i,t_j$ in the cover and a color $\sigma$.

It follows that
\begin{equation*}
\s(H,k,\ell) \geq \left(1-A(k) \left(\frac{k-1}{k}\right)^{\binom{n-2\ell}{r-\ell}}\right) \sum_{s \in \mathcal{S}(k)} \sum_{P \in \mathcal{P}_s}
\prod_{e \in E} \left( \sum_{t_m \subset e} s_i \right)
\end{equation*}
with $A(k)=4 \binom{|\mathcal{S}(k)|+1}{2}
\binom{k!}{2} \binom{c}{2}k$, as required. \end{proof}

We are now able to prove an auxiliary result that leads
to the proof of Theorem \ref{T2}.

\begin{lemma}\label{L4}
Let $k \geq 5$ be fixed.
Let $C=\{t_1,\ldots,t_c\}$ be an $\ell$-cover such that there exist
$i,j \in \{1,\ldots,c\}$, $i \neq j$, for which $|t_i \cap t_j| \geq 1$.
If  $|t_i \cup t_j| \leq r$, then there exists $n_0 > 0$ such that,
for $n \geq n_0$, the hypergraph
$H_{C,r}(n)$ is not extremal, i.e.,
$\kappa (H_{C, r}(n), k, \ell)  < \KC(n,r,k,l)$.
\end{lemma}


Before establishing this auxiliary result, we first argue that it leads to
the desired conclusion in Theorem \ref{T2},
at least in the case $k \not \equiv 1~(\bmod \, 3)$.


The lemma immediately implies that an extremal hypergraph $H^{\ast}$
on $n$ vertices is of the form $H_{C,r}(n)$, where $C=\{t_1,\ldots,t_c\}$ is such that, for each $i,j \in [c]$, $i \neq j$, either $t_i \cap t_j=\emptyset$ or $|t_i \cup t_j|>r$.

If $r \geq 2 \ell$ the latter condition cannot be satisfied, therefore it must be that $t_i$ and $t_j$ are disjoint, for every $i \neq j$, hence
the hypergraph  $H_{C,r}(n)$ is isomorphic to $H_{n,r,k,\ell}$. Moreover, if $r=2 \ell-1$, the condition $|t_i \cap t_j|=y>0$ implies that
$|t_i \cup t_j|=2\ell-y=r+1-y \leq r$, and therefore we must also have that all cover elements are disjoint in this case.

If $r< 2 \ell-1$, the condition $t_i \cap t_j=\emptyset$ tells us that $|t_i \cup t_j|=2 \ell>r$. In particular, the conditions $t_i \cap t_j=\emptyset$ or $|t_i \cup t_j|>r$ may be combined as $|t_i \cup t_j|>r$ for each $i,j \in [c]$, $i \neq j$. This yields our result.

When $k \equiv 1~(\bmod \, 3)$, this lemma also gives the structure of the
extremal hypergraph, but fails to determine whether the extremal
hypergraph has minimum cover size $c(k)$ or $c(k)-1$. This part is addressed at the end of the proof.

\begin{proof}[Proof of Lemma \ref{L4}] Fix $i$ and $j$ satisfying $t_i \cap t_j \neq \emptyset$ and $|t_i \cup t_j| = 2\ell - |t_i \cap t_j| \leq r.$

Let $U=\bigcup_{m=1}^c t_m$ and consider vertices $v \in t_i \cap t_j$ and $w \in [n] \setminus U$. Set $t_i'=t_i \bigtriangleup \{v,w\}$ and $C'=C \bigtriangleup \{t_i,t_i' \}$. When we think of $C'$ as an ordered set, we consider that $t_i'$ is the $i$-th element, while $t_1,\ldots,t_c$ have the position indicated by their index, as in $C$.

We claim that there exist $\delta>0$, $\xi \in (0,\gamma)$ and $n_0 \in \mathbb{N}$ such that
\begin{eqnarray} \label{e301}
\s(H_{C,r}(n),k,\ell) &\leq& \s(H_{C',r}(n),k,\ell)-\delta D(k)^{(1-\xi)
\binom{n-\ell}{r-\ell}}
\end{eqnarray}
for every $n \geq n_0$. Note that (\ref{e301})  implies our result, as
by Lemma \ref{L5}
we have
\begin{equation*}
\kappa(H_{C,r}(n),k,\ell) \leq \s(H_{C,r}(n),k,\ell) + D(k)^{(1-\gamma)\binom{n-N}{r-\ell}},
\end{equation*}
and
this may be rewritten as
\begin{equation}\label{mona1}
\begin{split}
\kappa(H_{C,r}(n),k,\ell) &\leq \s(H_{C',r}(n),k,\ell)-\delta D(k)^{(1-\xi)\binom{n-\ell}{r-\ell}} + D(k)^{(1-\gamma)\binom{n-N}{r-\ell}}\\
& \leq \kappa(H_{C',r}(n),k,\ell)-\delta D(k)^{(1-\xi)\binom{n-\ell}{r-\ell}} + D(k)^{(1-\gamma)\binom{n-N}{r-\ell}}.
\end{split}
\end{equation}
Lemma \ref{L4} now follows from the fact that
$$\delta D(k)^{(1-\xi)\binom{n-\ell}{r-\ell}} > D(k)^{(1-\gamma)\binom{n-N}{r-\ell}}$$
for $n$ sufficiently large, since $\xi<\gamma$ and $N \geq \ell$.

We now prove  inequality (\ref{e301}).
For simplicity, let $E=E(n)$ and $E'=E'(n)$ denote the sets of hyperedges of
$H=H_{C,r}(n)$ and $H'=H_{C',r}(n)$, respectively. Let $s=(s_1,\ldots,s_c)$
be an optimal solution to (\ref{eqL1}) and let $P=P(s)=(P_1,\ldots,P_c)$ be a
partition of the color set $[k]$ for which $|P_i|=s_i$, $i=1,\ldots,c$.
We define a function $\beta\colon E \longrightarrow \mathbb{N}$, where,
for $e \in E$, $\beta(e)=\beta_e=\sum_{t_i \subset e} s_i$.  Let $\beta'\colon
E' \longrightarrow \mathbb{N}$ be the analogous function
for $H_{C',r}(n)$. Consider the following families of $r$-subsets of $[n]$:
\begin{equation*}
\begin{array}{lll}
F_0&=&\{e \in [n]^r : t_i \cup t'_i \subseteq e\} \cup \{e \in [n]^r :
t_i,t'_i \not \subset e,  \exists g \neq i, ~t_g \subset e\}\\
F_1&=&\{e \in [n]^r : t_i \subset e, w \notin e, t_g \not\subset
e, \forall g \neq i\}\\
F'_1&=&\{e \in [n]^r : t'_i \subset e,  v \notin e, t_g \not\subset
e, \forall g \neq i\}\\
F_2&=&\{e \in [n]^r : t_i \subset e, w \notin e,  \exists g \neq i, ~t_g
\subset e\}\\
F'_2&=&\{e \in [n]^r : t'_i \subset e, v \notin e,  \exists g \neq i, ~t_g
\subset e\}.
\end{array}
\end{equation*}
Note that $E \cap E'=F_0 \cup F_2 \cup F'_2$, where the union is disjoint, while $E \setminus E'= F_1$ and $E' \setminus E= F'_1$. Moreover, by our definition of $\beta$ and $\beta'$, we have
\begin{equation*}
\begin{matrix}
\beta_e=s_i \textrm{ if }e \in F_1, & \beta'_{e}=\beta_e-s_i \geq 2
\textrm{ if }e \in F_2,\\
\beta'_{e}=s_i \textrm{ if }e \in F'_1, & \beta'_{e}=\beta_e+s_i
\geq 2+s_i \textrm{ if }e \in F'_2,\\
\beta'_e=\beta_e \textrm{ if }e \in F_0.&
\end{matrix}
\end{equation*}
By definition, we have
$$|\SC(H,P,k,\ell)|=\prod_{e \in E} \beta_e \textrm{ and }
|\SC(H',P,k,\ell)|=\prod_{e' \in E'} \beta'_{e'},$$
so that
\begin{eqnarray} \label{eq400}
 \frac{|\SC(H',P,k,\ell)|}{|\SC(H,P,k,\ell)|}
&=&
\frac{\left(\prod_{e' \in F'_1} s_i\right) \prod_{e \in E \cap E'} \beta'_e}{\left(\prod_{e \in F_1} s_i\right)
\prod_{e \in E \cap E'} \beta_e} \nonumber \\
&=&\frac{s_i^{|F'_1|-|F_1|}\left(\prod_{e \in F'_2} \beta'_e
\right)\left(\prod_{e \in F_2} \beta'_e \right)}{\left(\prod_{e \in F'_2}
\beta_e \right)\left(\prod_{e \in F_2} \beta_e \right)} \nonumber\\
&=&\frac{s_i^{|F'_1|-|F_1|}\left(\prod_{e \in F'_2} \beta'_e
\right)\left(\prod_{e \in F_2} \beta'_e \right)}{\left(\prod_{e \in F'_2}
(\beta'_e-s_i) \right)\left(\prod_{e \in F_2} (\beta'_e+s_i) \right)}.
\end{eqnarray}
Consider the function $\phi \colon F_1 \cup F_1'\cup F_2 \cup F_2'
\longrightarrow
F_1 \cup F_1'\cup F_2 \cup F_2'$ given by $\phi(e)=e \bigtriangleup \{v,w\}$.
It is easy to see that this function is its own inverse, in particular it is
injective. Moreover, our choice of $w \not\in U$ implies that
$\phi(F_1)\subseteq F_1'$ and $\phi(F_2') \subseteq F_2.$ Finally, observe that $\phi$ is a bijection between the sets $F'_1 \setminus \phi(F_1)$ and $F_2 \setminus \phi(F_2')$. To see why this last property is true, let $f'\in F'_1 \setminus \phi(F_1)$ and consider $f = \phi(f') =f'\bigtriangleup \{v,w\}$. Our choice of $f'$ implies that $f \notin F_1$, hence $f \in F_2$. However, $f \notin \phi(F_2')$ because $\phi(f)=f' \in F_1'$, so that $f\in F_2 \setminus \phi(F_2')$,
as claimed. The converse is analogous, and we infer that
\begin{eqnarray} \label{eq302}
&& |F_2|-|F_2'|=|F_2|-|\phi(F_2')|=|F_1'|-|\phi(F_1)|=|F_1'|-|F_1|
,  
\end{eqnarray}
and (\ref{eq400}) becomes
\begin{eqnarray} \label{eq401}
\frac{|\SC(H',P,k,\ell)|}{|\SC(H,P,k,\ell)|}&=&
\frac{s_i^{|F_2|-|F_2'|}\left(\prod_{e \in F'_2} \beta'_e
\right)\left(\prod_{e \in F_2} \beta'_e \right)}{\left(\prod_{e \in F'_2}
(\beta'_e-s_i) \right)\left(\prod_{e \in F_2} (\beta'_e+s_i) \right)}.
\end{eqnarray}

The following result is useful for our computations.
\begin{lemma}\label{L7}
Let $A=\{a_1,\ldots,a_p\}$ and $B=\{b_1,\ldots,b_q\}$ be sets of positive integers, let $m \in \{2,3,4\}$ and $M$ be positive integers, and suppose that $m+2 \leq a_i \leq M$, $2 \leq b_j \leq M$, for every $i$ and $j$, where
$q \geq \max\{p,1\}$. Let $\phi \colon[p]
\longrightarrow [q]$ be an injective function such that $a_i \leq b_{\phi(i)}+m$, for every $i \in [p]$. Then
\begin{eqnarray} \label{NEU1}
&& \frac{m^{q-p} \prod_{i \in A}a_i \prod_{j \in B} b_j}{\prod_{i
\in A}(a_i-m) \prod_{j \in B} (b_j+m)} \geq 1.
\end{eqnarray}
If $a_i< b_{\phi(i)}+m$, for some $i \in [p]$, then the right-hand side of
(\ref{NEU1})  
may be replaced by $1+\frac{m}{M^2-m^2}$. If $p < q$ and $\max\{m,b_j  : \;
j\in[q]\} \geq 3$, then the right-hand side of (\ref{NEU1})
 may be replaced by $\frac{6}{5}$.
\end{lemma}

\begin{proof}
We show Lemma \ref{L7} by induction on $p$.
If $p=0$, this amounts to proving that
$$m^{q} \prod_{j \in B} b_j \geq \prod_{j \in B} (b_j+m),$$
which is a consequence of $m\cdot b_j \geq b_j+m$, which holds for $m \geq 2$ and $b_{j} \geq 2$, for all $j \in [q]$.

Now, suppose that our result holds for $p-1 \geq 0$. Let $A=\{a_1,\ldots,a_p\}$,
$B=\{b_1,\ldots,b_q\}$, $k$ and $\phi$ be as in the statement of this lemma.
Let $i$ and $j=\phi(i)$ be such that $a_{i} \leq b_{j}+m$.
Then, with $A'=A\setminus\{a_i\}$, $B'=B \setminus \{b_j\}$ and $\phi'$ defined as the restriction of $\phi$ to $A'$, we have
\begin{equation*}
\begin{split}
m^{q-p} \prod_{i \in A}a_{i} \prod_{j \in B} b_{j}&=a_{i}b_{j}m^{q-p}
\prod_{i' \in A'}a_{i'} \prod_{j' \in B'} b_{j'}\\
& = \left((a_i-m)(b_j+m)+m(b_j+m-a_i)\right)  m^{(q-1)-(p-1)} \prod_{i' \in A'}a_{i'} \prod_{j' \in B'} b_{j'}\\
&\geq (a_i-m)(b_j+m) \prod_{i' \in A'}(a_{i'}-m) \prod_{j' \in B'} (b_{j'}+m)\\
&=\prod_{i \in A}(a_{i}-m) \prod_{j \in B} (b_{j}+m).
\end{split}
\end{equation*}
In the third line of the above equation, we are using the fact that $b_j+m-a_i \geq 0$ and the induction hypothesis. As a consequence,
$$\frac{m^{q-p} \prod_{i \in A}a_i \prod_{j \in B} b_j}{\prod_{i \in A}(a_i-m) \prod_{j \in B} (b_j+m)} \geq 1.$$
Moreover, if $b_j+m-a_i \geq 1$ for some $i \in [p]$ and $j=\phi(i)$, we have
\begin{equation*}
\begin{split}
\frac{a_ib_j}{(a_i-m)(b_j+m)}=1+\frac{m(b_j+m-a_i)}{(a_i-m)(b_j+m)}\geq 1 + \frac{m}{(M-m)(M+m)}.
\end{split}
\end{equation*}
Now, if $q>0$ and $\max\{m,b_{j} : \; j\in[q]\} \geq 3$, assume by symmetry
that $b_j \geq 3$. As in the base of induction, we replace $mb_j$ by $m+b_j$ and, using monotonicity and $m \geq 2$, we infer that
\begin{equation*}
\begin{split}
\frac{mb_j}{m+b_j}& \geq  \frac{3m}{m+3}
\geq \frac{6}{5},
\end{split}
\end{equation*}
as claimed.
\end{proof}

We now use Lemma \ref{L7} to evaluate  (\ref{eq401}). To this end,
let $A=F_2'$ and $B=F_2$, and, given $s =(s_1, \ldots , s_c)
 \in \mathcal{S}(k)$, define $a_{e'}=\beta'_{e'}$ for every $e'\in A=F_2'$ and $b_e=\beta'_e$
for every $e \in B=F_2$. It is clear that $s_{e'}+2 \leq a_{e'} \leq k$, as every element of $F_2'$ contains at least two elements in the $\ell$-cover of $H'$, one of them being $t_i'$, and $2 \leq b_e \leq k$, as every element of $F_2$ contains at least one element in the $\ell$-cover of $H'$. Thus we may set $m=s_i$ and $M=k$ in our application of Lemma \ref{L7}.

Let $\phi$ be again the bijection of $F_1 \cup F_1' \cup F_2 \cup F_2'$
on itself associating a hyperedge $e$ with $e \bigtriangleup \{v,w\}$,
which we have already seen to map $F_2'$ into $F_2$. If $e' \in F_2'$, we must
have $a_{e'} \leq b_{\phi(e')} + s_i$ because the only set in the cover
of $H'$ that covers $e'$ but does not cover $\phi(e')$ is $t_i'$,
by our choice of $w$.

We may apply Lemma \ref{L7} to (\ref{eq401}), for any partition $P \in \mathcal{P}_s$, and obtain
\begin{equation}\label{frida1}
\frac{|\SC(H',P,k,\ell)|}{|\SC(H,P,k,\ell)|}\geq 1.
\end{equation}
We now look at a particular solution $\hat{s} \in \mathcal{S}(k)$ for which
$\hat{s}_i=3$ for some $i$, which exists since $k\geq 5$.
Let $P=P(\hat{s})$ be a
partition of the color set, as before. We show that the inequality  (\ref{frida1}) becomes
 stronger in this case. To do this using Lemma \ref{L7}, we must show that, in the setting introduced above, we either have $|B|=|F_2|>|F_2'|=|A|$ or there exists $e' \in F_2'$ for which $a_e'<b_{\phi(e')}+\hat{s}_i$.

Consider an $r$-subset $f$ of $[n]$ such that  $f \cap (U \cup \{w\})=
t_i \cup t_j$, whose existence is guaranteed by our restriction
$|t_i \cup t_j| \leq r$ and by the fact that $n$ may be taken large enough so as to ensure the existence of sufficiently many elements outside $U$.
It is clear that $f \in F_2$, since $ w \not\in f$. Consider $f'=f \bigtriangleup \{v,w\} \in F'_1 \cup F'_2$. There are two cases:
\begin{itemize}
\item[(i)] if $f' \in F'_2$, then $\beta'_{f'}+\hat{s}_j-\hat{s}_i
\leq \beta'_{f}$, which implies $\beta'_{f'}<\beta'_{f}+\hat{s}_i$. This occurs because $t'_i \subset f', t'_i \not
\subset f$ and $t_j \not \subset f', t_j \subset f$, while our choice of $w$
also guarantees that, for $g \neq i$, there cannot be $t_g$ for which
$t_g \subset f'$ but $t_g \not \subset f$.

\item[(ii)] If $f'\in F'_1$,
then $|F'_1|>|F_1|$ because $\phi(f')=f$ does not lie in $F_1$. By (\ref{eq302})
this implies
that $q=|F_2|-|F_2'| = |F_1'|-|F_1| > 0$. Moreover, $\max \{m,b_e \; \vert \;
e \in B\} \geq \hat{s}_i \geq 3$.
\end{itemize}
As a consequence, we may apply Lemma \ref{L7} to (\ref{eq401})
and obtain for every $k \geq 5$:
\begin{eqnarray} \label{gap7}
\frac{|\SC(H',P(\hat{s}),k,\ell)|}{|\SC(H,P(\hat{s}),k,\ell)|}&\geq&
 \min\left\{\frac{6}{5},1+\frac{2}{k^2-16}\right\}>\frac{k^2-14}{k^2-15}.
\end{eqnarray}

Recall that, by Lemma \ref{L6}, we have
\begin{eqnarray}\label{tapsi8}
&& \s(H',k,\ell)-\s(H,k,\ell) \nonumber \\
&\geq& \left(1-A(k)\left(\frac{k-1}{k}\right)^{\binom{n-2\ell}{r-\ell}} \right)
\sum_{s \in \mathcal{S}(k)} \sum_{P' \in \mathcal{P}_s^{'}} |\SC(H',P',k,\ell)|
\nonumber \\
&& -
\sum_{s \in \mathcal{S}(k)} \sum_{P' \in \mathcal{P}_s^{'}} |\SC(H,P',k,\ell)|.
\end{eqnarray}
By our previous discussion, see (\ref{frida1}) and (\ref{gap7}),
 for $k \geq 5$, we have
\begin{eqnarray}\label{tapsi9}
&&\sum_{s \in \mathcal{S}(k)} \sum_{P' \in \mathcal{P}_s^{'}}
|\SC(H',P',k,\ell)|
- \sum_{s \in \mathcal{S}(k)} \sum_{P' \in \mathcal{P}_s^{'}}
|\SC(H,P',k,\ell)|
\nonumber \\
&\geq& \frac{1}{k^2-15}|\SC(H',P(\hat{s}),k,\ell)|
\nonumber \\
&\geq&  \frac{1}{k^2-15}  D(k)^{\binom{n-c\ell}{r-\ell}}.
\end{eqnarray}
The second to
last inequality may be derived with the same arguments used for establishing (\ref{tapsi6}). Also note that, given $\xi>0$ and $n$ sufficiently large, we have
\begin{equation}\label{tapsi10}
\begin{split}
D(k)^{\binom{n-c(k)\ell}{r-\ell}}&=D(k)^{\left(\prod_{a=1}^{r -
\ell}\frac{n-(c(k)-1)\ell-r+a}{n-r +a)}\right)\binom{n-\ell}{r-\ell}}
\geq D(k)^{(1-\xi)\binom{n-\ell}{r-\ell}},
\end{split}
\end{equation}
since
$$\lim_{n \rightarrow \infty} \prod_{a=1}^{r -
\ell}\frac{n-(c(k)-1)\ell-r+a}{n-r +a} = 1.$$
Thus we infer from (\ref{tapsi8})--(\ref{tapsi10}) that, for given $\xi > 0$
and $n$ sufficiently large,
\begin{eqnarray} \label{tapsi10a}
&& \s(H',k,\ell)-\s(H,k,\ell) \nonumber \\
&\geq&  \frac{1}{k^2 -15}
D(k)^{(1-\xi)\binom{n-\ell}{r-\ell}} -  A(k) \left(\frac{k-1}{k}\right)^{\binom{n-2\ell}{r-\ell}}
\sum_{s \in \mathcal{S}(k)} \sum_{P' \in \mathcal{P}_s^{'}}
|\SC(H',P',k,\ell)|.
\end{eqnarray}

For later use, fix $\xi<\min\{\frac{\gamma}{2},\frac{1}{3}\log_{D(k)}\frac{k}{k-1}\}$. By Lemmas \ref{L5} and \ref{L6}, for $n$ sufficiently large,
\begin{eqnarray}\label{tapsi11}
&& A(k) \left(\frac{k-1}{k}\right)^{\binom{n-2\ell}{r-\ell}}
\sum_{s \in \mathcal{S}(k)} \sum_{P' \in \mathcal{P}_s^{'}}
|\SC(H',P',k,\ell)| \nonumber \\
& \leq &
\frac{A(k)\left(\frac{k-1}{k}\right)^{\binom{n-2\ell}{r-\ell}}}{1-
A(k)\left(\frac{k-1}{k}\right)^{\binom{n-2\ell}{r-\ell}}} \, \s(H',k,\ell)
\nonumber \\
&\leq& 2A(k)\left(\frac{k-1}{k}\right)^{\binom{n-2\ell}{r-\ell}}
\kappa(H',k,\ell) \nonumber \\
& \leq& 2A(k)\left(\frac{k-1}{k}\right)^{\binom{n-2\ell}{r-\ell}}
\left(N(k)k^{\binom{\ell c(k)}{\ell+1}\binom{n-\ell-1}{r-\ell-1}}D(k)^{\binom{n-
\ell}{r-\ell}}\right) \nonumber \\
&=& K \left(\frac{k-1}{k}\right)^{\binom{n-2\ell}{r-\ell}}k^{\binom{\ell c(k)
}{\ell+1}\binom{n-\ell-1}{r-\ell-1}}D(k)^{\binom{n-\ell}{r-\ell}},
\end{eqnarray}
where $K = 2A(k)N(k)$ is independent of $n$ and,
in the second to last step, part (ii) of Theorem \ref{T1} is applied.
Note that, for $n$ sufficiently large, our choice of $\xi$ leads to
\begin{eqnarray}\label{tapsi12}
&& K \left(\frac{k-1}{k}\right)^{\binom{n-2\ell}{r-\ell}}k^{\binom{\ell c(k)
}{\ell+1}\binom{n-\ell-1}{r-\ell-1}}D(k)^{\binom{n-\ell}{r-\ell}} \nonumber\\
&=& K D(k)^{\binom{n-\ell}{r-\ell}\left(1-\left(\prod_{a=1}^{r-\ell}\frac{n-\ell-r+a}{n-2\ell+a}\right) \log_{D(k)}\frac{k}{k-1}+\frac{r-\ell}{n-\ell}\binom{c(k)\ell}{\ell+1}\log_{D(k)}k\right)} \nonumber \\
&\leq& D(k)^{(1-2\xi)\binom{n-\ell}{r-\ell}},
\end{eqnarray}
since
$$\lim_{n \rightarrow \infty} \prod_{a=1}^{r -
\ell}\frac{n-\ell-r+a}{n-2 \ell +a} = 1,$$
thus, with (\ref{tapsi11}) and (\ref{tapsi12}) we infer that
\begin{eqnarray} \label{tapsi10b}
A(k) \left(\frac{k-1}{k}\right)^{\binom{n-2\ell}{r-\ell}}
\sum_{s \in \mathcal{S}(k)} \sum_{P' \in \mathcal{P}_s^{'}}
|\SC(H',P',k,\ell)| &\leq& D(k)^{(1-2\xi)\binom{n-\ell}{r-\ell}}.
\end{eqnarray}

Combining (\ref{tapsi10a}) and
(\ref{tapsi10b}), we obtain, for $n$ sufficiently large, and fixed $k \geq 5$
\begin{eqnarray}\label{mona4}
 \s(H',k,\ell)-\s(H,k,\ell)
&\geq&  \frac{1}{k^2-15}
D(k)^{(1-\xi)\binom{n-\ell}{r-\ell}}-D(k)^{(1-2\xi)\binom{n-\ell}{r-\ell}}
\nonumber \\
&\geq& \delta D(k)^{(1-\xi)\binom{n-\ell}{r-\ell}},
\end{eqnarray}
for a  constant $\delta > 0$ and $\xi < \gamma$. This
proves  (\ref{e301}) and hence Lemma~\ref{L4}.
\end{proof}

An important feature of the proof of Lemma \ref{L4} is that it can be used to show more: $(C,r)$-complete hypergraphs whose cover is not of the form prescribed in Definition~\ref{optimal_hyper} are ``far'' from being optimal.  
\begin{remark}\label{mona3}
If $H_{C,r}(n)$ is $(C,r)$-complete, but $C$ is not of the form prescribed in Definition~\ref{optimal_hyper} for $k\geq 5$, then there exist $\delta>0$ and $n_0>0$ such that, for $n>n_0$,
$$\kappa(H_{C,r}(n),k,\ell)<(1-\delta)\KC(n,r,k,\ell).$$

Indeed, let $H_{C',r}(n)$ be an extremal hypergraph obtained by modifying the cover $C$ inductively, as in the proof of Lemma \ref{L4}, until the union of any two elements in the cover is larger than $r$, or they are disjoint.
By symmetry, we have that $|\SC(H_{C',r}(n),P,k,\ell)|$ is the same for
every optimal solution $s \in \mathcal{S}(k)$ and every $P \in \mathcal{P}_s$. In particular, Lemmas \ref{L5} and \ref{L6} imply
$$|\SC(H_{C',r}(n),P,k,\ell)| \geq \frac{1}{N(k)}\left(\kappa(H_{C',r}(n),
k,\ell)-D(k)^{\binom{n-N}{r-\ell}(1-\gamma)}\right),$$
where $N(k)$ is given in Definition \ref{defCND} and
$\gamma$ is the positive constant defined in (\ref{otto3}). Moreover, as $\kappa(H_{C',r}(n),k,\ell) \geq D(k)^{\binom{n-\ell c(k)}{r-\ell }}$ by (\ref{tapsi6}), we have
$$\lim_{n \rightarrow \infty}
\frac{D(k)^{\binom{n-N}{r-\ell}(1-\gamma)}}{\kappa(H_{C',r}(n),k,\ell)}
\leq \lim_{n \rightarrow \infty} D(k)^{(1-\gamma)\binom{n-N}{r-\ell}
-\binom{n-\ell c(k)}{r-\ell )} }=0,$$
hence, for $n$ sufficiently large,
$|\SC(H_{C',r}(n),P,k,\ell)| \geq \frac{1}{2 N(k)}\kappa(H_{C',r}(n),k,\ell).$
With this, for $n$ sufficiently large,
(\ref{tapsi9}) may be modified to
$$\sum_{s \in \mathcal{S}(k)} \sum_{P' \in \mathcal{P}_s^{'}}
|\SC(H',P',k,\ell)|
- \sum_{s \in \mathcal{S}(k)} \sum_{P' \in \mathcal{P}_s^{'}}
|\SC(H,P',k,\ell)| \geq \frac{1}{2(k^2-15)N(k)}\kappa(H_{C',r}(n),k,\ell),$$
so that equation (\ref{mona4}) becomes
$$\s(H_{C',r}(n),k,\ell)-\s(H_{C,r}(n),k,\ell) \geq
  \frac{1}{4(k^2-15)N(k)}\kappa(H_{C',r}(n),k,\ell),$$
and,
as a consequence,
\begin{equation*}
\begin{split}
\kappa(H_{C,r}(n),k,\ell) &\leq \s(H_{C,r}(n),k,\ell)+D(k)^{\binom{n-N}{r-\ell}(1-\gamma)}\\
&\leq \left(1-\frac{1}{8(k^2-15)N(k)}\right) \kappa(H_{C',r}(n),k,\ell)\\
& \leq (1-\delta)\KC(n,r,k,\ell)
\end{split}
\end{equation*}
for any $\delta \leq \frac{1}{8(k^2-15)N(k)}$, concluding our claim.
\end{remark}

We now finish the proof of Theorem \ref{T2} in the case
$k \equiv 1~(\bmod \, 3)$, $k \geq 5 $. Recall that we already know that an extremal hypergraph has the structure described in the statement of the theorem, but
we need to determine whether the extremal hypergraph
has minimum cover size $c(k)$ or $c(k)-1$. Using Lemmas \ref{L5} and \ref{L6}, it suffices to show that, if the cover has size $c(k)$, the number of star colorings of the corresponding hypergraph is substantially larger than when the cover has size $c(k)-1$. Actually, we do not count the number of star-colorings exactly, but determine it asymptotically through the sum of the numbers of colorings given in the statement of Lemma~ \ref{L6}. For brevity, we shall drop the reference to asymptotics, and just write, that we are counting star colorings.

In the following we distinguish two cases according to
the relation of $r$ and $2\ell-1$.

\subsection{The Case  $r  <  2 \ell-1$}\label{case1}

Let $H^{\ast}_0=H_{C,r}(n)$ be a $(C,r)$-complete hypergraph
on $n$ vertices
with $\ell$-cover $C=\{t_1,\ldots,t_{c(k)}\}$ such that $|t_i \cup t_j| > r$,
for every $i,j \in [ c(k)]$, $i \neq j$.
Let $H^{\ast}_1=H_{C',r}(n)$ be the analogous hypergraph for the
$\ell$-cover $C'=\{t'_1,\ldots,t'_{c(k)-1}\}$ with the same property.
Note that the requirement $r  <  2 \ell-1$
implies that each hyperedge in both $H^{\ast}_0$
and $H^{\ast}_1$ is covered by exactly one element of the cover.
In particular, given any optimal solution $s \in \mathcal{S}_0 \subset\mathcal{S}(k)$, where $\mathcal{S}_0$ contains all optimal solutions with two 2's, and any partition $P \in \mathcal{P}_s$, we must have
\begin{equation}\label{eqQL}
|\SC(H^{\ast}_0,P,k,\ell)|=3^{(c(k)-2)\binom{n-\ell}{r-\ell}}2^{2\binom{n-\ell}{r-\ell}},
\end{equation}
since each set in the $\ell$-cover covers exactly $\binom{n-\ell}{r-\ell}$ hyperedges
and no hyperedge is covered more than once. Analogously, for an optimal solution $s'\in \mathcal{S}_1 \subset\mathcal{S}(k)$, where $\mathcal{S}_1$
contains all optimal solutions with one $4$,
and any partition $P' \in \mathcal{P}_{s'}$, we have
$$|\SC(H^{\ast}_1,P',k,\ell)|=3^{(c(k)-2)\binom{n-\ell}{r-\ell}}4^{\binom{n-\ell}{r-\ell}}=|\SC(H^{\ast}_0,P,k,\ell)|.$$
Moreover, for $s \in \mathcal{S}_0$, we have
\begin{eqnarray} \label{eq501}
|\mathcal{S}_0||\mathcal{P}_s|&=&\binom{c(k)}{2}\frac{k!}{2 \cdot 2 \cdot
(3!)^{c(k)-2}
},
\end{eqnarray}
while, for $s' \in \mathcal{S}_1$,
\begin{eqnarray} \label{eq502}
|\mathcal{S}_1||\mathcal{P}_{s'}|&=& (c(k)-1)\frac{k!}{4! \cdot
(3!)^{c(k)-2}},
\end{eqnarray}
so that, by Lemma \ref{L5}, for $n$ sufficiently large,
\begin{eqnarray}\label{mona5}
\frac{\kappa(H^{\ast}_0,k,\ell)}{\kappa(H^{\ast}_1,k,\ell)}&>&
\frac{\s(H^{\ast}_0,k,\ell)}{2\s(H^{\ast}_1,k,\ell)}
\geq \frac{|\mathcal{S}_0||\mathcal{P}_s|}{2|\mathcal{S}_1||\mathcal{P}_{s'}|} = \frac{3c(k)}{2}>1,
\end{eqnarray}
implying that the extremal hypergraphs have $\ell$-cover of size $c(k)$.

\subsection{The Case  $r \geq 2 \ell-1 $}\label{case2}
Let $H^{\ast}_0$ and $H^{\ast}_1$ be the complete
$n$-vertex hypergraphs with minimum
$\ell$-covers of size $c(k)$ and $c(k)-1$, respectively,
where the elements in each of the two covers are mutually disjoint.
As in the previous case,
we can show that, regardless of the solution
$s \in \mathcal{S}_0$ and the partition $P \in \mathcal{P}_s$ chosen,
the value of $|\SC(H^{\ast}_0,P,k,\ell)|$ is the same, since any mapping of the
vertices of the hypergraph that interchanges two sets in the $\ell$-covers
but keeps the remaining vertices intact is an isomorphism. The same is
true for optimal solutions $s' \in \mathcal{S}_1$ and partitions
$P \in \mathcal{P}_{s'}$.

Since $k \geq 5$ we have  $c=c(k) \geq 3$.

Consider the $r$-uniform hypergraph $H_0^{\ast}=H_{C,r}(n) =(V,E)$
on $n$ vertices and with $\ell$-cover
  $ C= \{t_1, \ldots , t_{c}\}$,
for pairwise disjoint
$\ell$-subsets of $V$.
We determine the number of star colorings of $H_0^{\ast}$.
Let  $s \in \mathcal{S}_0 \subset\mathcal{S}(k)$
be an optimal solution,
where $\mathcal{S}_0$ contains all optimal solutions with two $2$'s,
and let  $P \in \mathcal{P}_s$ be any partition.
Assume that the $\ell$-sets $t_{c-1}$ and $t_c$ correspond to the two $2$'s.
Let $X = \{ t_1, \ldots , t_{c-2}\}$.

Any hyperedge $e$, which contains both $\ell$-sets  $t_{c-1}$ and $t_{c}$,
and with  $[e]^{\ell}
\cap \{t_1, \ldots , t_{c-2 } \} = \{ t_i \; | \; i \in I \}$
can be colored with $(4 + |I|)$ colors, and for a fixed set
$I$ the number of these
hyperedges is
\begin{eqnarray} \label{a737}
A(|I|) &=& \sum_{i=0}^{c-2-|I|} (-1)^i \binom{ n - \ell (|I| +2 +i)}{r -
\ell (|I| + 2 +i)} \binom{c-2-|I|}{i},
\end{eqnarray}
which follows by inclusion-exclusion, as $  \binom{ n - \ell (|I| +2 +i)}{r -
\ell (|I| + 2 +i)} \binom{c-2-|I|}{i}$ counts the number of $r$-sets,
which contain the $\ell$-sets $t_{c-1}$, $t_c$, and $t_j$, $j \in I$,
as well as $i$ further
$\ell$-sets from the $\ell$-cover $C$.

Any hyperedge $e$, which contains exactly one of the $\ell$-sets
  $t_{c-1}$ or $
t_{c}$, say $t_{c}$, and with $[e]^{\ell}
\cap \{t_1, \ldots , t_{c-2 } \} = \{ t_j \; | \; j \in J \}$
can be colored with $( 2 + 3 |J| )$ colors, and for each of the $\ell$-sets
$t_{c-1}$ and $t_{c}$, and, again by inclusion-exclusion,
for a fixed set $J$ the number of these
hyperedges is
\begin{eqnarray} \label{b737}
B(|J|) &=& \sum_{i=0}^{c-1-|J|} (-1)^i \binom{ n - \ell (|J| +1 +i)}{r -
\ell (|J| + 1 +i)} \binom{c-1-|J|}{i}.
\end{eqnarray}

Any hyperedge $e$ with $[e]^{\ell} \cap \{ t_{c-1}, t_{c} \} = \emptyset$
and with $[e]^{\ell}
\cap \{t_1, \ldots , t_{c-2} \} = \{ t_k \; | \; k \in K \}$
can be colored with $3 |K| $ colors, and
for a
fixed set $K$ the number of these hyperedges is
\begin{eqnarray} \label{c737}
C(|K|) &=& \sum_{i=0}^{c-|K|} (-1)^i \binom{ n - \ell (|K|  +i)}{r -
\ell (|K|  +i)} \binom{c-|K|}{i}.
\end{eqnarray}

Let $q = \lfloor r/\ell \rfloor$.
Then, given the partition $P$ the  number
of  star colorings
of $H_0^{\ast}$ is
\begin{eqnarray} \label{neweq601}
&& \s(H_0^{\ast},P,k,\ell) \nonumber \\
&=& \left( \prod_{x=0}^{\min (c-2,q-2)} \prod_{I \in
\binom{X
}{x}} (4 +3x)^{
A(x)} \right) \cdot
\left( \prod_{y=0}^{\min (c-2, q-1)} \prod_{J \in \binom{X
}{y}}
( 2 + 3y)^{B(y)} \right)^2 \cdot \nonumber \\
&& \cdot \left(
\prod_{z=1}^{\min (c-2,q)} \prod_{K \in \binom{X
}{z}} (  3z)^{
C(z)} \right).
\end{eqnarray}

On the other hand,
consider the $r$-uniform hypergraph $H_1^{\ast}= H_{C',r}(n)=(V,E)$
on $n$ vertices   with $\ell$-cover
$C'= \{ t_1, \ldots , t_{c-1} \}$,   $r \geq 2\ell - 1 $,
for pairwise disjoint
$\ell$-subsets of $V$. As above,  we determine the number of
star colorings of $ H_{1}^{\ast}$.
Let $s' \in \mathcal{S}_1 \subset\mathcal{S}(k)$ be an optimal solution,
where $\mathcal{S}_1$ contains all optimal solutions with one $4$,
and let  $P' \in \mathcal{P}_{s'}$ be any partition.
Assume that  the set $t_{c-1}$  corresponds to the one $4$ in $s'$.

Every hyperedge $e$, which contains  the $\ell$-set $t_{c-1}$
and with  $[e]^{\ell}
\cap \{t_1, \ldots , t_{c-2 } \} = \{ t_i \; | \; i \in I \}$
can be colored with $(4 + 3 |I|)$ colors,
and, by inclusion-exclusion,
for a
fixed set $I$ there
are
\begin{eqnarray} \label{d737}
D(|I|) &=& \sum_{i=0}^{c-2-|I|} (-1)^i \binom{ n - \ell (|I|+1  +i)}{r -
\ell (|I|+1  +i)} \binom{c-2-|I|}{i}
\end{eqnarray}
of these hyperedges.

Every hyperedge $e$ with $t_{c-1} \not\subseteq  e $
and with $[e]^{\ell}
\cap \{t_1, \ldots , t_{c-2 } \} = \{ t_k \; | \; k \in K \}$
can be colored with $  3 |K|$ colors, and
again by inclusion-exclusion,
for a
fixed set $K$ there
are
\begin{eqnarray} \label{e737}
E(|K|) &=& \sum_{i=0}^{c-1-|K|} (-1)^i \binom{ n - \ell (|K|  +i)}{r -
\ell (|K|  +i)} \binom{c-1-|K|}{i}
\end{eqnarray}
of these hyperedges.

Thus,  given the partition $P'$,
 the  number
of these star colorings
of $H_1^{\ast}$ is
\begin{eqnarray} \label{neweq602}
\s(H_1^{\ast}, P',k,\ell) &=& \left(
\prod_{x=0}^{\min ( c-2, q-1)} \prod_{I \in \binom{X}{x}} (4 +3x)^{
D(x)} \right)
\left( \prod_{z=1}^{\min ( c-2,q )} \prod_{K \in \binom{X } {z}} (  3 z)^{
E(z)} \right) .
\end{eqnarray}

To finish the proof, it suffices to show that $\s(H_0^{\ast}, P,k,\ell)$ is at least as big as $\s(H_1^{\ast}, P',k,\ell)$, from which we may derive as in (\ref{mona5}) that $\kappa(H_0^{\ast},k,\ell)>\kappa(H_1^{\ast},k,\ell)$. However, with the exception of the case $r = 2\ell - 1$
where we have $\s(H_0^{\ast}, P,k,\ell) = \s(H_1^{\ast}, P',k,\ell)$, the proof of this result involves calculations of reasonable length and is included in the Appendix.

This concludes the proof of Lemma \ref{L4}, and
finishes the proof of Theorem \ref{T2}.  \end{proof}

\subsection{The case $k=4$ revisited}

To finish this section, we address the case $k=4$. First, we give a proof of Theorem \ref{T2_4}. We then provide a more precise result, which describes precisely the extremal hypergraphs in this case.

\begin{proof}[Proof of Theorem \ref{T2_4}] By the first part of the proof of Theorem \ref{T2} (see Lemma \ref{124lemmmanew}), we already know that, for $n$ sufficiently large, an extremal
 hypergraph $H^{\ast}$ on $n$ vertices is of the form $H_{C,r}(n)$,
for some set $C$ of $\ell$-subsets of $[n]$. However, as we are in the case $k \equiv 1~(\bmod \, 3)$, our restriction
is $|C| \in \{1,2\}$. We want to show that two is the correct size of $C$.

Consider the $r$-uniform hypergraph $H_0^{\ast}=H_{C,r}(n)=(V_0,E_0)$
on $n$ vertices and with $\ell$-cover  $C= \{t_1,  t_{2}\}$
for $\ell$-subsets of $V_0$ with $|t_1 \cap t_2| = y$.

There are $\binom{4}{2}$ possibilities to distribute the $4$ colors in sets of size $2$ to the $\ell$-subsets $t_1$ and $t_2$.
The number of $r$-subsets of $V_0$ containing the set $t_1$ and not $t_2$
is  $\binom{n-\ell }{r - \ell} - \binom{n - 2\ell + y}{r - 2 \ell +y}$, and vice versa.
The number of $r$-subsets of $V_0$ containing both sets $t_1$ and $t_2$ is
$\binom{n-2\ell + y }{r - 2\ell + y} $, and these can be colored with $4$ colors.
Hence,  the number of star colorings  of $H_0^{\ast}$ is
\begin{eqnarray} \label{klm2}
\binom{4}{2}  2^{2\left( \binom{n-\ell }{r - \ell} - \binom{n - 2\ell + y}{r
- 2 \ell +y}
\right)}  4^{\binom{n - 2\ell + y}{r - 2 \ell +y}}
&=& 6 \cdot  4^{ \binom{n-\ell  }{r - \ell}},
\end{eqnarray}
and is independent of the intersection of the sets
in the $\ell$-cover.

On the other hand, consider the $r$-uniform hypergraph
$H_1^{\ast}=H_{C,r}(n) =(V_1,E_1)$
on $n$ vertices and with $\ell$-cover
$ C= \{t_1\}$
for
an $\ell$-subset of $V_1$.
Every hyperedge can be colored with $4$ colors, hence the number of colorings
of $H_1^{\ast}$ is
$$
4^{\binom{n - \ell}{r - \ell}}.
$$

Using (\ref{klm2}), and by
Lemma \ref{L5} we have, for $n$ sufficiently large, that
\begin{eqnarray*}
\frac{\kappa(H^{\ast}_0,k,\ell)}{\kappa(H^{\ast}_1,k,\ell)}
&>& \frac{\s(H^{\ast}_0,k,1)}{2\s(H^{\ast}_1,k,1
)}
=
\frac{6}{2} > 1,
\end{eqnarray*}
implying that for $k=4$
the extremal hypergraphs have $\ell$-cover of size $c(4)=2$.
\end{proof}

In light of Theorems \ref{T2_4} and \ref{T2}, the hypergraph $H_{n,r,k,\ell}$ given in Definition \ref{def_extremal} is the unique extremal hypergraph for $\KC(n,r,k,\ell)$ whenever $n$ is sufficiently large and we are in one of the cases $k \leq 3$, $k=4$ and $\ell=1$, or $k \geq 5$ and $r \geq 2 \ell-1$. Furthermore, even in the cases for which uniqueness is not obtained, the hypergraph $H_{n,r,k,\ell}$ is listed as candidate for extremal hypergraph. This naturally raises the question of whether $H_{n,r,k,\ell}$ is always extremal for sufficiently large $n$. However, this is not true, as implied by the following result in the case $k=4$. Its proof relies on a more careful counting of the number of Kneser colorings in the ``candidate extremal hypergraphs'' of Theorem \ref{T2_4}.

\begin{theorem} \label{theo1_inter}
Let $r$ and $\ell \geq 2$ be integers with $\ell < r$. Then, there
exists $n_0$, such that for all $n \geq n_0$,
the extremal hypergraph for $\mathcal{P}_{n,r,4,\ell}$ is the hypergraph
$H_{C,r}(n)$
where $C=\{t_1,t_2\}$ is an $\ell$-cover such that $|t_1 \cap t_2|=\ell-1$.  
\end{theorem}

We use the following notation.
\begin{definition}
Let $H = (V, E)$ be a hypergraph with $\ell$-cover $C = \{ t_1, \ldots,
t_c \}$. A \emph{generalized star coloring} of $H$ is a Kneser coloring such that, for every color $\sigma$, all the hyperedges of $H$ with color $\sigma$ contain some fixed element $t_i=t_i(\sigma)$ of the cover.  
A Kneser coloring of $H$ that is not a generalized star coloring is called
a \emph{non-star coloring}.
\end{definition}

Note that a generalized star coloring is just a relaxation of star colorings in that the number of colors assigned to the cover sets need not be given by an optimal solution to (\ref{eqL1}).

\begin{proof}
Let $H^{\ast} =(V,E)$ be a complete, $r$-uniform hypergraph
on $n$ vertices with
 $\ell$-cover $C = \{ t_1, t_2 \}$ where
$|t_1 \cap t_2| = y \leq \ell -1$. To compute the number of generalized star colorings, observe that we may either assign two colors to each cover element, or three colors to one cover element and one color to the other. The former can be done in $\binom{4}{2}$ ways and the latter in $2 \binom{4}{3}$ ways. The number $S(y)$ of
generalized star colorings of $H^{\ast}$ is therefore
given by
\begin{eqnarray}
S(y) &=& \binom{4}{2}  2^{2  (\binom{n - \ell}{r - \ell}
- \binom{n-2\ell + y}{r - 2\ell + y})}
4^{\binom{n-2\ell + y}{r - 2\ell + y}} +
\label{FGH2a}\\
&&  + 2  \binom{4}{3}  
\left( 3^{\binom{n - \ell}{r - \ell}- \binom{n-2\ell + y}{r - 2\ell + y}}
- 3 \cdot  2^{\binom{n - \ell}{r - \ell}- \binom{n-2\ell + y}{r - 2\ell + y}}
+3 \right)
4^{\binom{n-2\ell + y}{r - 2\ell + y}}
\nonumber \\
&=& 6 \cdot 4^{\binom{n - \ell}{r - \ell}} + \label{FGH2} \\
&&  + 8 \cdot
 3^{\binom{n - \ell}{r - \ell}}
 \left( \frac{4}{3} \right)^{\binom{n-2\ell + y}{r - 2\ell + y}}
 \left( 1 - 3  
\left( \frac{2}{3} \right)^{\binom{n - \ell}{r - \ell}-
\binom{n-2\ell + y}{r - 2\ell + y}} + 3    \left( \frac{1}{3}
\right)^{\binom{n - \ell}{r - \ell}- \binom{n-2\ell + y}{r - 2\ell + y}}
\right),
\nonumber \end{eqnarray}
where
the powers of $4$ arise from the $r$-sets that contain $t_1 \cup t_2$,
which can be colored by any of the $4$ colors. As usual, we use $\binom{n}{i} = 0$
for  integers $i<0$.
The term
$
3^{\binom{n - \ell}{r - \ell}- \binom{n-2\ell + y}{r - 2\ell + y}}
- 3 \cdot  2^{\binom{n - \ell}{r - \ell}- \binom{n-2\ell + y}{r - 2\ell + y}}
+3
$
counts the number of $3$-colorings of all $r$-subsets that contain $t_1$,
but not $t_2$, or vice versa,
for which all three colors are used, as the other
colorings are already counted in (\ref{FGH2a}) by the term
$\binom{4}{2}  2^{2  (\binom{n - \ell}{r - \ell}
- \binom{n-2\ell + y}{r - 2\ell + y})}$.

Then, for $n$ sufficiently large, we have $S(y-1) \leq S(y)$, hence $S(y)$  
is maximal for $y = \ell - 1$. By Pascal's identity $\binom{n}{k}=\binom{n-1}{k}+\binom{n-1}{k-1}$, we infer, for $n$ sufficiently large, that
\begin{eqnarray}   \label{FGH3}
S(\ell - 1) &=& 6 \cdot   4^{\binom{n - \ell}{r - \ell}} + 8 \cdot
 3^{\binom{n - \ell-1}{r - \ell}}
4^{\binom{n-\ell - 1}{r - \ell -1}}
\cdot \left( 1 - 3    \left( \frac{2}{3} \right)^{\binom{n - \ell-1}{r - \ell}}
+  3   \left( \frac{1}{3}
\right)^{\binom{n - \ell-1}{r - \ell}} \right) \nonumber \\
 &\geq&  6 \cdot   4^{\binom{n - \ell}{r - \ell}} + 4 \cdot
 3^{\binom{n - \ell-1}{r - \ell}}  4^{\binom{n-\ell - 1}{r -
 \ell -1}}.
\end{eqnarray}

We now show that, if $|t_1 \cap t_2|< \ell-1$, then $H^{\ast}$ is not extremal. To this end, we find an upper bound on the number of non-star
colorings of $H^{\ast}$. For such a coloring $\Delta$, there exists at least one pair $(a,b)$ of
$r$-sets $a, b \in E$ of the same color such that $|a \cap b| \geq \ell$ with $t_1 \subset a$ and $t_2 \subset b$, but $t_1 \cup t_2 \not\subseteq a$ and $t_1 \cup t_2 \not\subseteq b$.
 Let $|a \cap (t_2 \setminus t_1)| =q$
and $|b \cap (t_1 \setminus  t_2)
 | = p$,  where we may assume that $p \leq q$ by
symmetry. Thus
 \begin{eqnarray} \label{FGH4}
p + y \leq \ell - 1 \; \; \; &\mbox{and} &\; \; \; q + y \leq \ell -1.
\end{eqnarray}

Let
\begin{eqnarray*}
{\mathcal F}_1(b) &=& \{ e \in E \, | \, t_1 \subset e, t_2
\not\subset e, \Delta (e) = \Delta (b)
\} \label{heu1} \\
{\mathcal F}_2 (a) &=& \{ e \in E \, | \, t_2 \subset e,
t_1 \not\subset e, \Delta (e) = \Delta (a)
\}. \label{heu2}
\end{eqnarray*}

\begin{lemma}  \label{lemmacap}
There exist constants $C_1, C_2 > 0$ and a positive integer $n_0$ such that, for all $n \geq n_0$, we have
\begin{eqnarray}
| {\mathcal F}_1(b)
 | &\leq&  C_1 \cdot  n^{r-2\ell +p+y}  \hspace{1cm}
\mbox{and} \hspace{1cm} | {\mathcal F}_2(a)
| \leq  C_2 \cdot  n^{r-2\ell +q+y}.  \label{laber2}
\end{eqnarray}
\end{lemma}

Notice that  
$
 r \geq 2\ell -p-y$   and $r \geq 2\ell - q - y
$.
Indeed, since $a$ and $b$ intersect in at least $\ell$ elements, we must have $|(a \setminus t_1) \cap (b \setminus t_1)|
\geq \ell - (p + y)$,
hence $a$ must contain at least
$
\ell + (\ell - (p+y)) = 2\ell -p-y
$
elements, which implies that $r \geq 2\ell -p -y$. Similarly we obtain
$ r \geq  2\ell -q - y$.

\begin{proof}
Given $r$-sets $a$ and $b$, there is a constant $C_1 > 0$ such that
\begin{eqnarray*} \label{FGH5}
|{\mathcal F}_{1}(b) | \leq \binom{ r - p - y}{ \ell - (p+y)}
 \binom{ n - (2\ell - p - y)}{r - (2\ell - p - y)}
&\leq& C_1  n^{r-2\ell +p+y},
\end{eqnarray*}
as we can choose $\ell - (p+y)$ elements from the set $b \setminus t_1$ in $
\binom{ r - p - y}{ \ell - (p+y)}$ ways, and the remaining $r - (2\ell - p - y)$
elements in at most $ \binom{ n - (2\ell - p - y)}{r - (2\ell - p - y)}$ ways.
Similarly, the second inequality follows, namely
\begin{eqnarray*} \label{FGH5a}
|{\mathcal F}_{2}(a)| \leq  \binom{ r - q - y}{ \ell - (q+y)}
\cdot \binom{ n - (2\ell - q - y)}{r - (2\ell - q - y)}
&\leq& C_2 \cdot  n^{r-2\ell +q+y}.
\end{eqnarray*}
\end{proof}

Moreover, with the $r$-sets $a$ and $b$ fixed, the subfamilies ${\mathcal G}_1 \subseteq {\mathcal F}_{1}(b)$ and ${\mathcal G}_2 \subseteq  {\mathcal F}_{2}(a)$ may be assigned the same color as  $a$ and $b$, provided that $|a' \cap b'| \geq \ell$ for all $r$-sets $a' \in {\mathcal G}_1$ and $b' \in {\mathcal G}_2$.

Note that the $r$-sets  $a$ and $b$ may be chosen in at most
$\binom{n- \ell }{r - \ell }^2 \leq  n^{2r-2\ell}$ ways, the subfamilies  ${\mathcal G}_1$ and ${\mathcal G}_2$ may be fixed in at most $2^{|{\mathcal F}_{1}(b)| + |{\mathcal F}_{2}(a)|}$ ways and one of the colors may be chosen in four ways. As $p \leq q$ and using Lemma \ref{lemmacap}, we deduce that the total number of choices is at most
\begin{eqnarray} \label{FGH6}
 4   \cdot  n^{2  r - 2  \ell }
2^{|{\mathcal F}_{1}(b)| + |{\mathcal F}_{2}(a)|} &\leq&
4 \cdot  n^{2  r - 2  \ell }  2^{C_1 n^{r-2\ell +p+y}
+ C_2 n^{r-2\ell +q+y}} \nonumber
\\
&\leq& 4  \cdot  n^{2  r - 2  \ell }  2^{2C_2 n^{r-2\ell +q+y}} .
\end{eqnarray}

Having fixed the color of $a$ and $b$ and all $r$-sets in  ${\mathcal G}_1 \cup   {\mathcal G}_2$,
we may use this color only for $r$-sets covering $t_1 \cup t_2$, which appear if $r \geq 2\ell - y$. We use the remaining three colors for a star coloring of the set of  uncolored hyperedges, which can be colored in at most
\begin{eqnarray} \label{FGH7}
&&  
2   \binom{3}{2}
  2^{\binom{n - \ell}{r - \ell} - \binom{n - 2\ell +y}{r-2\ell+y} }
 4^{\binom{n-2\ell + y}{r-2\ell + y}}  
= 6 \cdot 2^{\binom{n - \ell}{r - \ell}
 + \binom{n-2\ell + y}{r-2\ell + y}}
\end{eqnarray}
ways; hence, with (\ref{FGH6}) and (\ref{FGH7}),
for $n$ sufficiently large, the number of such non-star colorings is at most
\begin{eqnarray} \label{FGH7b}
24 \cdot  n^{2  r - 2  \ell }
 2^{2C_2 n^{r-2\ell +q+y}}  2^{\binom{n - \ell}{r - \ell}
+ \binom{n-2\ell + y}{r-2\ell + y}}.
\end{eqnarray}

On the other hand, assume that there exists another pair $(a_1,b_1) \neq (a,b)$ of $r$-sets $a_1, b_1 \in E$  with $|a_1 \cap b_1| \geq \ell$ and
$t_1 \subset a_1$, and $t_2 \subset b_1$, but $t_1 \cup t_2 \not\subseteq a_1$ and  $t_1 \cup t_2 \not\subseteq b_1$,
and with $|C\cap (t_2\setminus t_1) | = q'$ and $|D\cap (t_1 \setminus t_2)| = p'$,
say $p' \leq q'$, where $\Delta (a_1) = \Delta (b_1) \neq \Delta (a)$.
Let  ${\mathcal F}_1(b_1)$ and ${\mathcal F}_2(a_1)$
be defined  as in (\ref{heu1}) and (\ref{heu2}).
By  Lemma \ref{lemmacap} with $p' \leq q'$, for $n$ sufficiently large,
we have that, for some constants $C_1', C_2' > 0$,
\begin{eqnarray} \label{FGH7C}
|{\mathcal F}_1(b_1)| + |{\mathcal F}_2(a_1)|
&\leq & 2^{C_1' n^{r-2\ell +p'+y}}
 + 2^{C_2' n^{r-2\ell +q'+y}} \leq 2^{2C_2' n^{r-2\ell +q'+y}}.
\end{eqnarray}
Combining this with (\ref{FGH6})
and (\ref{FGH7C}) gives us at most
\begin{eqnarray} \label{FGH8}
&& 4
\cdot n^{2r - 2\ell}   2^{2C_2 n^{r-2\ell +q+y}}
\cdot 3 \cdot 2 \cdot n^{2r - 2\ell}  
2^{2C_2' n^{r-2\ell +q'+y}}
 4^{\binom{r - 2\ell + y}{r - 2\ell + y}}
\nonumber
\\
&=&  24  \cdot  n^{4r - 4\ell}  2^{  2C_2 n^{r-2\ell +q+y} +
2 C_2' n^{r-2\ell +q'+y} }  4^{\binom{r - 2\ell + y}{r - 2\ell + y}}
\end{eqnarray}
such non-star colorings, since the two colors not used so far can be taken for those remaining hyperedges not covering $t_1 \cup t_2$ in at most two ways. Indeed, if we use one of the two remaining colors
for another pair $(a_2,b_2)$ of $r$-sets that is distinct
from the pairs $(a,b)$ and $(a_1,b_1)$, where $t_1 \cup t_2 \not\subseteq a_2$ and  $t_1 \cup t_2 \not\subseteq b_2$, then by
Lemma~\ref{lemmacap} with (\ref{FGH4}), for some constant $C' > 0$
this color can be used for at most
$C' \cdot  n^{r - \ell - 1}$ $r$-sets containing $t_1$ but not $t_2$, or
$t_2$ but not $t_1$, respectively. However, this leaves
at least $ \binom{n - \ell}{r - \ell} - C' n^{r - \ell -1}$
uncolored $r$-sets containing $t_1$ or $t_2$,
but not $t_1 \cup t_2$, which cannot be colored properly with a single color.  

Hence, combining (\ref{FGH7b}) and (\ref{FGH8}) with the inequalities $q +y \leq \ell - 1$ and $q' + y \leq \ell - 1$ given in (\ref{FGH4}),
we have that, for some constant $C > 0$, the total number of non-star colorings of $H^{\ast}$ has the upper bound  
\begin{eqnarray} \label{FGH9}
&&  
24  \cdot   n^{2  r - 2  \ell }  
2^{\binom{n - \ell}{r - \ell}
+ \binom{n-2\ell + y}{r-2\ell + y}}
2^{2C_2 n^{r - 2\ell +q+y}}  
 +   24  \cdot n^{4  r - 4  \ell}   2^{2C_2 n^{r-2\ell +q+y} +
2C_2' n^{r-2\ell +q'+y} }  4^{\binom{r - 2\ell + y}{r - 2\ell + y}}
\nonumber \\
&\leq& 24  \cdot 2^{C n^{r - \ell -1}}  \left(  n^{2  r - 2  \ell }
 2^{\binom{n - \ell}{r - \ell}
+ \binom{n-\ell -1}{r-\ell -1}}   +  
n^{4  r - 4  \ell}  
 4^{\binom{r - \ell -1}{r - \ell - 1}} \right).
\end{eqnarray}

For fixed $p$ and $q$, respectively $p'$ and $q'$, the number of possibilities for choosing the intersections
$|a \cap (t_2 \setminus t_1)| = q$ and $|b \cap (t_1 \setminus t_2)| = p$ as well as
$|a_1 \cap (t_2 \setminus t_1)| = q'$ and $|b_1 \cap (t_1 \setminus t_2)| = p'$
is bounded from above by a constant, and affect the upper bound
(\ref{FGH9}) by at most a constant factor. Notice that, for $n$ sufficiently large, the term
\begin{eqnarray*} \label{FGH9a}
&&
 4 \cdot
 3^{\binom{n - \ell -1}{r - \ell }}  
 4^{\binom{n-\ell -1}{r - \ell -1}}
\end{eqnarray*}
in (\ref{FGH3}) is  much larger than the upper
bound in (\ref{FGH9}).
Therefore, the number of Kneser colorings of $H^{\ast}$ is maximized
for $y = \ell - 1$, which finishes the proof of Theorem \ref{theo1_inter}.
\end{proof}

\section{The asymptotic behavior of $\KC(n,r,k,\ell)$}\label{sec_asy}

We now use the knowledge of properties of extremal hypergraphs
obtained in Section~\ref{sec_extremal}
to derive the asymptotic behavior $\KC(n,r,k,\ell)$. First note the following easy observation.
\begin{lemma}\label{Lql}
Let $r \geq 2$, $k \geq 4$ and $1 \leq \ell < r$ be given and let $H=([n],E)$
be a hypergraph as in Defintion~\ref{optimal_hyper}. Let $\mathcal{S}(k)$ be the set of optimal solutions to (\ref{eqL1}) and, given $s \in \mathcal{S}(k)$, consider the set $\mathcal{P}_s$ of all ordered partitions
$P(s)=(P_1,\ldots,P_c)$ of the set $[k]$ of colors such that $|P_i|=s_i$,
for every $i \in [c]$. Then
\begin{eqnarray} \label{star29}
\alpha(n,r,k,\ell)&=&\sum_{s \in \mathcal{S}(k)} \sum_{P \in \mathcal{P}_s}  \prod_{e \in E} \left( \sum_{t_i \subset e} s_i \right)
\end{eqnarray}
is independent of the choice of $H$.
\end{lemma}

\begin{proof} If $k \geq 5$ and $r \geq 2 \ell-1$, this is immediate, since the extremal hypergraphs are uniquely defined up to isomorphism. If $r<2 \ell-1$, the result follows because, by Lemma \ref{L4} in any extremal hypergraph, every hyperedge is covered by exactly one cover element.

If $k=4$, this is a consequence of the proof of Theorem \ref{T2_4}, namely of the discussion preceeding equation (\ref{klm2}). \end{proof}

Now, putting this observation together with Lemmas \ref{L5} and \ref{L6}, we obtain the asymptotic behavior of the function $\KC(n,r,k,\ell)$.
\begin{theorem}\label{T3}
Let $r \geq 2$, $k \geq 4$ and $1 \leq \ell < r$ be given. Then, there exist a function $f_{r,k,\ell}=f_{r,k,\ell}(n)$ and an integer $n_0>0$ such that, for the function $\alpha(n,r,k,\ell)$ defined in Lemma \ref{Lql},  
\begin{itemize}
\item[1.] $|\KC(n,r,k,\ell)-\alpha(n,r,k,\ell)|<f_{r,k,\ell}(n)$ for $n>n_0$;

\item[2.] $\displaystyle{\lim_{n \rightarrow \infty} \frac{f_{r,k,\ell}(n)}{\KC(n,r,k,\ell)}=0.}$
\end{itemize}
In particular, the function $\KC(n,r,k,\ell)$ is asymptotically equal
 to $\alpha(n,r,k,\ell)$.
\end{theorem}

\begin{proof} For a fixed $n$, let $H$ be an extremal hypergraph for $\KC(n,r,k,\ell)$, so that $\kappa(H,k,\ell)=\KC(n,r,k,\ell)$. Recall the definition
of $\alpha(n,r,k,\ell)$ from  (\ref{star29}).

By Lemma \ref{L5}, there is an integer $n_0'>0$ such that, for $n \geq n_0'$, the function $g_{r,k,\ell}=g_{r,k,\ell}(n)=D(k)^{\binom{n-\ell}{r-\ell}(1-\gamma)}$, where $\gamma>0$ is independent of $n$, satisfies
\begin{eqnarray} \label{star30}
|\kappa(H,k,\ell)-\s(H,k,\ell)| &\leq& g_{r,k,\ell}(n),
\end{eqnarray}
where $\s(H,k,\ell)$ denotes the number of star colorings of $H$. Moreover,  it is clear that
$$\displaystyle{\lim_{n \rightarrow \infty} \frac{g_{r,k,\ell}(n)}{\KC(n,r,k,\ell)} \leq \lim_{n \rightarrow \infty} \frac{g_{r,k,\ell}(n)}{\s(H,k,\ell)}=0.}$$

On the other hand, Lemma \ref{L6} gives an integer $n_0''>0$ such that, for $n \geq n_0''$, the function $h_{r,k,\ell}=h_{r,k,\ell}(n)=A\left(1-1/k\right)^{\binom{n-2\ell}{r-\ell}}$, where $A$ is independent of $n$, satisfies
\begin{eqnarray} \label{star31}
|\s(H,k,\ell)-\alpha(n,r,k,\ell)| &\leq&
h_{r,k,\ell}(n)\alpha(n,r,k,\ell).
\end{eqnarray}
By Lemma \ref{L6}, we also have
$$(1-h_{r,k,\ell}(n))\alpha(n,r,k,\ell) \leq \KC(n,r,k,\ell),$$
from which we deduce that
$$\displaystyle{\lim_{n \rightarrow \infty}
\frac{h_{r,k,\ell}(n)\alpha(n,r,k,\ell)}{\KC(n,r,k,\ell)}\leq
\lim_{n \rightarrow \infty}
\frac{h_{r,k,\ell}(n)\alpha(n,r,k,\ell)}{(1-
h_{r,k,\ell}(n))\alpha(n,r,k,\ell)}=0},$$
since $h_{r,k,\ell}(n)$ tends to zero as $n$
tends to infinity,
With the triangle inequality, the result now follows from (\ref{star30})
and (\ref{star31}) with $n_0= \max\{n_0',n_0''\}$ and
$f_{r,k,\ell}=g_{r,k,\ell}+h_{r,k,\ell}\alpha(n,r,k,\ell)$.

As a consequence, given $r \geq 2$, $k \geq 4$ and $1 \leq \ell < r$, the number $\KC(n,r,k,\ell)$ is asymptotically equal to $\alpha(n,r,k,\ell)$. \end{proof}

Before finding a formula for $\alpha(n,r,k,\ell)$, we combine the results obtained so far to prove the remaining theorems stated in the introduction.

\begin{proof}[Proof of Theorem \ref{Tint1}]
Let $f(n)=\frac{1}{\KC(n,r,k,\ell)}(f_{r,k,\ell}(n)+g_{r,k,\ell}(n)+h_{r,k,\ell}(n))$,
with $f_{r,k,\ell}(n)$, $g_{r,k,\ell}(n)$ and $h_{r,k,\ell}(n)$ defined in
the proof of Theorem \ref{T3}. Our result follows easily from parts (1) and (2) of Theorem \ref{T3}, and from (\ref{star30}) and (\ref{star31}) with $H$
replaced by $H_{n,r,k,\ell}$, both of which hold because $H_{n,r,k,\ell}$ is one of the hypergraphs in the extremal family described in Definition~\ref{optimal_hyper}.
\end{proof}  

\begin{proof}[Proof of Theorems \ref{Tint2_cont} and \ref{Tint2_cont4}]
The two theorems are proved simultaneously, as they rely on the same arguments. For the first assertion, Theorem~\ref{T3} implies that, as $n$ tends to
infinity, $\KC(n,r,k,\ell)$ is asymptotically equal to $\alpha(n,r,k,\ell)$.
By Lemmas \ref{L5}, \ref{L6} and \ref{Lql}, the function $\alpha(n,r,k,\ell)$
is in turn asymptotically equal to $\kappa(H,k,\ell)$ for any $H$ defined
in Definition~\ref{optimal_hyper}. The result follows.

For the converse, let $H$ be an $r$-uniform hypergraph on $[n]$ with minimum $\ell$-cover $C$. First consider that $k \not\equiv 1~(\bmod \, 3)$. If $|C| \neq c(k)$, $\delta$ is a constant in the interval $(0,1/2)$ and $n$ is sufficiently large, then $\kappa(H,k,\ell)<\delta \KC(n,r,k,\ell)$ by Remark \ref{rem_cover}.

Now, assume that $|C|=c(k)$, but $H$ is not $(C,r)$-complete in the sense of Definition \ref{def_extremal}. Let $e$ be a set covered by $C$ that is not a hyperedge in $H$ and consider $H'=H \cup \{e\}$. By the definition of star colorings, it is clear that $|\SC(H,k,\ell)|\leq \frac{1}{2} |\SC(H',k.\ell)|$,
as each star coloring of $H$ can be extended to a star coloring of $H'$ by assigning to $e$ any color associated with one of the cover elements contained in $e$, and there are at least two such colors. Now, since the set of colorings that are not star colorings is small by Lemma \ref{L5}, we have that, for any $\nu>0$, $\kappa(H,k,\ell) \leq (\frac{1}{2}+\nu) \KC(n,r,k,\ell)$ if $n$ is sufficiently large.

Finally, assume that $H$ is $(C,r)$-complete, but $C$ is not an $\ell$-cover
as in the definition of $\mathcal{H}_{r,k,\ell}(n)$.
Remark \ref{mona3} tells us that there are $\delta>0$ and $n_0>0$ such that, for $n \geq n_0$, $\kappa(H,k,\ell) \leq (1-\delta) \KC(n,r,k,\ell)$. Therefore,
 Theorem \ref{Tint2} holds with $\eps_0=\min\{1/3,\delta\}$ in this case.

If $k \equiv 1~(\bmod \, 3)$, the arguments above can be used, but it remains to prove that, if a $(C,r)$-complete hypergraph $H_{C,r}(n)$
 has minimum $\ell$-cover size $c(k)-1$, it has substantially fewer colorings than a hypergraph in $\mathcal{H}_{r,k,\ell}(n)$. However, this is an immediate consequence of the calculations in Section \ref{case1} (see equation (\ref{mona5})) and in Section \ref{case2} (see  (\ref{mona6}) and (\ref{mona7})).
\end{proof}

\begin{proof}[Proof of Theorem \ref{Tint3}]
For $k \geq 4$, this theorem is an easy consequence of Theorem \ref{Tint2_cont}. Indeed, if either $k=4$ and $\ell=1$ or $k \geq 5$ and $r \geq 2\ell-1$, Theorem \ref{Tint2_cont} implies that there are $\eps_0>0$ and $n_0>0$ such that there is a unique, up to isomorphism, $r$-uniform hypergraph $H$ on $[n]$, $n>n_0$, for which $\kappa(H,k,\ell)>(1-\eps_0)\KC(n,r,k,\ell)$. Stability follows trivially, and in a strong form, as, for any $\eps>0$, the constant $\delta=\eps_0$ is such that, whenever $\kappa(H,k,\ell)>(1-\delta)\KC(n,r,k,\ell)$, we have $|E(H)\bigtriangleup E(H')|=0< \eps$ for some extremal hypergraph $H'$.

Now, if $k=4$ and $\ell>1$, or $k \geq 5$ and $r < 2\ell-1$, define
the $r$-uniform hypergraph $H'_{n,r,k,\ell}=H_{C,r}(n)$, where
the cover $C=\{t_1,\ldots,t_{c(k)}\}$ of $\ell$-subsets of $[n]$ is such that
$$\left|\bigcap_{i=1}^{c(k)} t_i \right|=2\ell-r-1.$$
We observe that, for $k \geq 5$, the case under consideration always has $2\ell-r-1 \geq 1$. Although this is not true for $k=4$, the same argument holds if we let the intersection of all
 cover elements be $\ell-1$. By Theorem \ref{Tint2_cont} we infer that
$$\lim_{n \rightarrow \infty} \frac{\kappa(H_{n,r,k,\ell},k,\ell)}{\KC(n,r,k,\ell)}=\lim_{n \rightarrow \infty} \frac{\kappa(H'_{n,r,k,\ell},k,\ell)}{\KC(n,r,k,\ell)}=1,$$
with $H_{n,r,k,\ell}$ given in Definition \ref{def_extremal}.
On the other hand, in $H'_{n,r,k,\ell}$ every two hyperedges are $(2\ell-r-1)$-intersecting, while at least $(c(k) - 1)
\binom{n - c(k)\ell}{r - \ell}$ hyperedges $e$ in $H_{n,r,k,\ell}$ are disjoint from at least
$$(c(k)-1)\binom{n-r-\ell}{r-\ell} \geq K_1 \cdot n^{r-\ell}$$
hyperedges of $H_{n,r,k,\ell}$, where $K_1$ is a constant. Since the number of hyperedges in an extremal hypergraph is bounded above by $c(k)\binom{n-r-\ell}{r-\ell} \leq K_2 \cdot n^{r-\ell}$ for some constant $K_2$,
and we cannot turn $H_{n,r,k,\ell}$ into $H'_{n,r,k,\ell}$ with the removal or addition of fewer than $K_1 \cdot
n^{r-\ell}$ hyperedges, the problem $\mathcal{P}_{n,r,k,\ell}$ is not stable.  

For $k=3$, the result follows easily from Remark \ref{rem_cover} and from an argument analogous to the one dealing with hypergraphs that are not complete with respect to their minimum cover in the proof of Theorem \ref{Tint2}. Similar results may be easily proven for $k=2$ when $n>n_0$ sufficiently large, and stability also follows in this case.
\end{proof}

\subsection{The exact value of $\alpha(n,r,k,\ell)$}\label{eq_alpha}

Note that, as we determined the extremal hypergraphs in Section~\ref{sec_extremal}, the value of $\alpha(n,r,k,\ell)$  was calculated in the case when $k \equiv 1~(\bmod \, 3)$. Indeed
\begin{itemize}
\item[(a)] $\displaystyle{\alpha(n,r,4,\ell)=6 \cdot  4^{ \binom{n-\ell }{r - \ell}}}$ (see equation (\ref{klm2})).

\item[(b)] if $r<2 \ell$ and $k \geq 5$, $\displaystyle{\alpha(n,r,k,\ell)=
\binom{c(k)}{2}\frac{k!}{2 \cdot 2 \cdot (3!)^{c(k)-2}
}3^{(c(k)-2)\binom{n-\ell}{r-\ell}}2^{2\binom{n-\ell}{r-\ell}}}$ (see equations (\ref{eqQL}) and (\ref{eq501})).

\item[(c)] if $r \geq 2 \ell$ and $k \geq 5$,
\begin{eqnarray}
&& \alpha(n,r,k,\ell) \nonumber \\
&=& \binom{c(k)}{2}\frac{k!}{(3!)^{c(k)-2} \cdot 2
\cdot 2} \cdot \left( \prod_{x=0}^{\min (c-2,q-2)} \prod_{I \in
\binom{X
}{x}} (4 +3x)^{
A(x)} \right)  \nonumber \\
&& \times \left( \prod_{y=0}^{\min (c-2, q-1)} \prod_{J \in \binom{X
}{y}}
( 2 + 3y)^{B(y)} \right)^2 \cdot \left(
\prod_{z=1}^{\min (c-2,q)} \prod_{K \in \binom{X
}{z}} (  3z)^{
C(z)} \right),
\end{eqnarray}
where $q=\lfloor r/\ell \rfloor$, $A(x)$, $B(y)$ and $C(z)$ are given by (\ref{a737}), (\ref{b737}) and (\ref{c737}), respectively, and $X$ has size $c(k)-2$. This follows from equation (\ref{neweq601}).
\end{itemize}
Observe that, in the above, the formula for the case $r=2\ell-1$ is given in conjunction with the case $r<2\ell-1$, as there are no hyperedges containing more than one cover element. It is easy to extend these calculations to general values of $k$, i.e., if $k=4$ or $r < 2 \ell$, then
$$\alpha(n,r,k,\ell)=N(k)D(k)^{\binom{n-\ell}{r-\ell}},$$
with $N(k)$ and $D(k)$ given in Definition \ref{defCND}.

If $r\geq 2 \ell$, the expression for $\alpha(n,r,k,\ell)$ is
\begin{itemize}
\item[(i)] for $k \equiv 0 ~(\bmod \, 3)$,
\begin{eqnarray*}
& \alpha(n,r,k,\ell)
&= N(k) \prod_{z=1}^{\min (c(k),q)} \prod_{K \in \binom{X
}{z}} (3z)^{
C(z)},
\end{eqnarray*}
\item[(ii)] for $k \equiv 1 ~(\bmod \, 3)$,
\begin{eqnarray*}
&\alpha(n,r,k,\ell)
=& N(k) \left( \prod_{x=0}^{\min (c(k)-2,q-2)} \prod_{I \in
\binom{X}{x}} (4 +3x)^{
A(x)} \right)  \nonumber \\
&& \times \left( \prod_{y=0}^{\min (c(k)-2, q-1)} \prod_{J \in \binom{X
}{y}}
( 2 + 3y)^{B(y)} \right)^2 \cdot \left(
\prod_{z=1}^{\min (c(k)-2,q)} \prod_{K \in \binom{X
}{z}} ( 3z)^{
C(z)} \right),
\end{eqnarray*}
\item[(iii)] for $k \equiv 2 ~(\bmod \, 3)$,
\begin{eqnarray*}
&\alpha(n,r,k,\ell)
=& N(k) \left( \prod_{y=0}^{\min (c(k)-1, q-1)} \prod_{J \in \binom{X
}{y}}
( 2 + 3y)^{B(y)} \right) \cdot \left(
\prod_{z=1}^{\min (c(k)-1,q)} \prod_{K \in \binom{X
}{z}} (  3z)^{
C(z)} \right),
\end{eqnarray*}
\end{itemize}
where $q=\lfloor r/\ell \rfloor \geq 2$ and $X$ has size $c(k)=k/3$, if $k \equiv 0~(\bmod \, 3)$, $c(k)-2=\lfloor k/3 \rfloor -1$, if $k \equiv 1~(\bmod \, 3)$, and  $c(k)-1=\lfloor k/3 \rfloor$, if $k \equiv 2~(\bmod \, 3)$. The functions
$A(x)$, $B(y)$ and $C(z)$ are given by
\begin{eqnarray*}
A(x) &=& \sum_{i=0}^{c(k)-2-x} (-1)^i \binom{ n - \ell (x +2 +i)}{r -
\ell (x + 2 +i)} \binom{c(k)-2-x}{i}\\
B(y) &=& \sum_{i=0}^{c(k)-1-y} (-1)^i \binom{ n - \ell (y +1 +i)}{r -
\ell (y + 1 +i)} \binom{c(k)-1-y}{i}\\
C(z) &=& \sum_{i=0}^{c(k)-z} (-1)^i \binom{ n - \ell (z  +i)}{r -
\ell (z  +i)} \binom{c(k)-z}{i}.
\end{eqnarray*}

\section{Open problems and concluding remarks}\label{sec_final}

In this paper, we have addressed the problem $\mathcal{P}_{n,r,k,\ell}$ of
determining $\KC(n,r,k,\ell)$, the largest number of $(k,\ell)$-Kneser colorings over all $r$-uniform hypergraphs on $n$ vertices. This has been fully solved in the following cases:
\begin{itemize}
\item[(1)] for any value of $n$, $r$ and $\ell$, if $k=2$;

\item[(2)] for any value of $r$ and $\ell$, if $k \in \{3,4\}$ and $n$ is sufficiently large;

\item[(3)] for any value of $r \geq 2 \ell-1$, if $k \geq 5$ and $n$ is sufficiently large.
\end{itemize}
Moreover, we have described precisely the extremal hypergraphs in each of these cases. In particular, when $n$ is sufficiently large, the restriction of this problem to graphs, namely $\mathcal{P}_{n,2,k,1}$,  has been solved completely.

For all remaining values of $r$, $k$ and $\ell$, we found the problem $\mathcal{P}_{n,r,k,\ell}$ to be unstable in the sense of Definition \ref{def_stable}. Notwithstanding, we have determined the asymptotic value of $\KC(n,r,k,\ell)$, as well as the family of asymptotically extremal hypergraphs, that is, the family of $r$-uniform hypergraphs $H=H(n)$ on $[n]$ such that, for every $\eps>0$, there exists $n_0$ such that, for $n>n_0$, the inequality $\kappa(H,k,\ell) \geq (1-\eps)\KC(n,r,k,\ell)$ holds.

The hypergraph $H_{n,r,k,\ell}$ given in Definition \ref{def_extremal} plays an important role, as it is the unique extremal hypergraph for $\mathcal{P}_{n,r,k,\ell}$ whenever $n$ is sufficiently large and the problem is stable. Furthermore, even when the problem is unstable, we have established that $H_{n,r,k,\ell}$ is asymptotically optimal. However, Theorem~\ref{theo1_inter} implies that $H_{n,r,k,\ell}$ is not optimal in the case $k=4$. This behavior is not accidental, and it is possible to show that, for $n$ sufficiently large, the hypergraph $H_{n,r,k,\ell}$ is never extremal when $\mathcal{P}_{n,r,k,\ell}$ is unstable. To prove this, one may compare the number of colorings of $H_{n,r,k,\ell}$ with the number of colorings in a $(C,r)$-complete hypergraph with the right cover size for which every two cover elements have intersection of size $ 2\ell -r -1$, and show that the latter has more Kneser colorings. We conjecture that the following stronger result is true.
\begin{conjecture}\label{conj1}
If $k \geq 5$, $r$ and $\ell$ are positive integers with $\ell < r < 2 \ell$, then a hypergraph $H=H_{C,r}(n)$ such that
$$\kappa(H,k,\ell)=\KC(n,r,k,\ell)$$
must satisfy $|C|=c(k)=\lceil k/3 \rceil$ and $|t_i \cap t_j|=2\ell -r -1$ for every distinct $t_i,t_j \in C$.
\end{conjecture}
Note that, even if Conjecture \ref{conj1} is true, there may be several
 configurations of $\ell$-sets in $C$ whose pairwise intersections
 have size $2\ell - r - 1$, depending on the size of $C$. Therefore it might be
of interest to investigate which of these configurations yield the largest number of Kneser colorings. Results in this direction would probably be useful in determining whether optimal configurations for the problem $\mathcal{P}_{n,r,k,\ell}$ are always unique up to isomorphisms.

\section{Appendix}

This appendix presents calculations needed to finalize the
analysis of the
case $r \geq 2\ell -1$ in the proof of Theorem \ref{T2}. Recall that our
aim is to show that $\s(H_0^{\ast}, P,k,\ell)$  is at least as big as
 $\s(H_1^{\ast}, P',k,\ell)$, with $\s(H_0^{\ast}, P,k,\ell)$ and $\s(H_1^{\ast}, P',k,\ell)$
given by equations (\ref{neweq601}) and (\ref{neweq602}), respectively. Recall that $q := \lfloor r/\ell \rfloor \geq 1$.

We infer from
(\ref{a737})--(\ref{neweq602})
that
\begin{eqnarray} \label{newjj1}
&& \frac{\s(H_0^{\ast}, P,k,\ell)}{\s(H_1^{\ast}, P',k,\ell)} \nonumber \\
&=& \frac{ \prod_{x=0}^{\min ( c-2, q-2)} \prod_{I \in
\binom{X
}{x}} (4 +3x)^{
A(x)}
}{\prod_{x=0}^{\min (c-2, q-1)} \prod_{I \in \binom{X}{x}} (4 +3x)^{
D(x)} }
\cdot  \frac{\prod_{z=1}^{\min (c-2, q)} \prod_{K \in \binom{X
}{z}} (  3z)^{
C(z)}}{\prod_{z=1}^{\min (c-2,q) } \prod_{K \in \binom{
X}{z}} (  3 z)^{
E(z)}}
 \nonumber \\
&& \times
\left( \prod_{y=0}^{\min (c-2,q-1)} \prod_{J \in \binom{X
}{y}}
( 2 + 3y)^{B(y) } \right)^2 \nonumber \\
&=& \frac{ \left( \prod_{y=0}^{\min ( c-2, q-1)}
( 2 + 3y)^{2
B(y) \binom{c-2}{y}} \right) }
{\left( \prod_{z=1}^{\min (c-2,q)} (  3 z)^{
( E(z) -
C(z) ) \binom{c-2}{z}  }\right)
} \cdot \frac{ \prod_{x=0}^{\min ( c-2, q-2)}
 (4 +3x)^{
A(x) \binom{c-2}{x}}
}{\prod_{x=0}^{\min (c-2, q-1)}  (4 +3x)^{
D(x)\binom{c-2}{x}} }   \nonumber \\
&=& \frac{ \left( \prod_{y=1}^{\min ( c-2, q-1)}
( 2 + 3y)^{2
B(y) \binom{c-2}{y}} \right) }
{\left( \prod_{z=1}^{\min (c-2, q)} (  3 z)^{
( E(z) -
C(z) )\binom{c-2}{z}  }\right)
} \cdot \frac{ \prod_{x=1}^{\min ( c-2, q-2)}
 (4 +3x)^{
A(x) \binom{c-2}{x}}
}{\prod_{x=1}^{\min (c-2, q-1)}  (4 +3x)^{
D(x) \binom{c-2}{x}} }    \nonumber \\
&&  \times
   \frac{2^{2 B(0) }}{4^{
  D(0) - A(0) }}.  
\end{eqnarray}

\begin{lemma}
For every non-negative integer  $x$, we have
\begin{eqnarray} \label{klm1}
B(x) = D(x) - A(x) \hspace{1cm} &\mbox{and}& \hspace{1cm} B(x) = E(x) - C(x).
\end{eqnarray}
\end{lemma}

\begin{proof}
 Consider the sum $A(x) + B(x)$. By (\ref{a737}) and (\ref{b737}), this counts
the number of hyperedges $e$, which contain both $\ell$-sets  $t_{c-1}$ and
$t_{c}$ or which contain only the $\ell$-set $t_{c-1}$,
and with  $[e]^{\ell}
\cap \{t_1, \ldots , t_{c-2 } \} = \{ t_i \; | \; i \in I \}$
 for a fixed $x$-set
$I$, which by (\ref{d737}) is equal to $D(x)$. Similarly, the identity
$B(x) + C(x)
= E(x)$ follows.

Alternatively, one can see this with (\ref{a737}), (\ref{b737}),  (\ref{d737}),
and by Pascal's identity,
namely,
\begin{eqnarray*}
&& B(x) - (D(x) - A(x))\\
&=& \sum_{i=0}^{c-1-x} (-1)^{i}
\binom{n - \ell (x+i+1)}{r - \ell (x+i+1)} \binom{c-1-x}{i}\\
&& -
\sum_{i=0}^{c-2-x} (-1)^{i}
\binom{n - \ell (x+i+1)}{r - \ell (x+i+1)} \binom{c-2-x}{i}\\
&& +
\sum_{i=0}^{c-2-x} (-1)^{i}
\binom{n - \ell (x+i+2)}{r - \ell (x+i+2)} \binom{c-2-x}{i} \\
&=&
\sum_{i=1}^{c-1-x} (-1)^{i}
\binom{n - \ell (x+i+1)}{r - \ell (x+i+1)} \binom{c-2-x}{i-1}\\
&& -
\sum_{i=1}^{c-1-x} (-1)^{i}
\binom{n - \ell (x+i+1)}{r - \ell (x+i+1)} \binom{c-2-x}{i-1}
= 0.
\end{eqnarray*}
Moreover, by (\ref{b737}), (\ref{c737}),   (\ref{e737}),
and, again by Pascal's identity,
we infer that
\begin{eqnarray*}
&& B(x) - (E(x) - C(x))\\
&=& \sum_{i=0}^{c-1-x} (-1)^{i}
\binom{n - \ell (x+i+1)}{r - \ell (x+i+1)} \binom{c-1-x}{i} \\
&& -
\sum_{i=0}^{c-1-x} (-1)^{i}
\binom{n - \ell (x+i+1)}{r - \ell (x+i+1)} \binom{c-1-x}{i}
= 0,
\end{eqnarray*}
which shows (\ref{klm1}).
\end{proof}

Now, (\ref{newjj1}) and (\ref{klm1}) give
\begin{eqnarray} \label{anewjj1}
&& \frac{\s(H_0^{\ast}, P,k,\ell)}{\s(H_1^{\ast}, P',k,\ell)}
= \frac{  \prod_{y=1}^{\min ( c-2, q-1)}
( 2 + 3y)^{2
B(y) \binom{c-2}{y}}  }
{ \prod_{z=1}^{\min (c-2, q)} (  3 z)^{
B(z) \binom{c-2}{z}
}
} \cdot \frac{ \prod_{x=1}^{\min ( c-2, q-2)}
(4 +3x)^{
A(x) \binom{c-2}{x}
}}
{\prod_{x=1}^{\min (c-2, q-1)}  (4 +3x)^{
D(x) \binom{c-2}{x}} }.
\end{eqnarray}

In the following
we distinguish two cases, depending on the order of $c$ and $q$.

\subsection{The Case  $c  \leq q$}
In this case,
equality
 (\ref{anewjj1}) becomes with (\ref{klm1})
\begin{eqnarray} \label{aanewjj1}
&& \frac{\s(H_0^{\ast}, P,k,\ell)}{\s(H_1^{\ast}, P',k,\ell)} \nonumber \\
&=& \frac{  \prod_{y=1}^{ c-2}
( 2 + 3y)^{2
B(y) \binom{c-2}{y}}  }
{\left( \prod_{z=1}^{c-2} (  3 z)^{
B(z) \binom{c-2}{z}  
) } \right)
\cdot  \left(
\prod_{x=1}^{c-2}  (4 +3x)^{(D(x) - A(x)) \binom{c-2}{x}
 } \right) }  \nonumber \\
&=& \left(
\prod_{x=1}^{c-2}  \frac{
( 2 + 3x)^{2
B(x) \binom{c-2}{x}}  }
{  (  3 x)^{
B(x) \binom{c-2}{x}
  } \cdot
(4 +3x)^{
B(x) \binom{c-2}{x}  
  }
   }
  \right) \nonumber \\
&=& \prod_{x=1}^{c-2}  \left(\frac{
( 2 + 3x)^{2
}  }
{  3 x
(4 +3x)
   } \right)^{B(x) \binom{c-2}{x}}.
\end{eqnarray}
With $ 1 \leq x \leq c-2$, we have by monotonicity
$$
\frac{(2+3x)^2}{ 3x( 4 + 3x)} = 1 + \frac{4}{12x + 9x^2}
\geq 1 + \frac{4}{(3c-2)(3c-6)}
$$
for $c  \geq 3$, hence (\ref{aanewjj1}) becomes
\begin{eqnarray} \label{newjj3}
 \frac{\s(H_0^{\ast}, P,k,\ell)}{\s(H_1^{\ast}, P',k,\ell)}
&\geq&
\prod_{x=1}^{c-2}  \left( 1 + \frac{4}{(3c-2)(3c-6)}  \right)^{
B(x) \binom{c-2}{x}}  \nonumber \\
&=&
\left( 1 + \frac{4}{(3c-2)(3c-6)}  \right)^{ \sum_{x=1}^{c-2}
B(x) \binom{c-2}{x}}  
> 1.
\end{eqnarray}

Using (\ref{eq501}) and (\ref{eq502}), and by
Lemma \ref{L5} we obtain from (\ref{newjj3}) for $c \geq 3$
\begin{eqnarray}\label{mona6}
\frac{\kappa(H^{\ast}_0,k,\ell)}{\kappa(H^{\ast}_1,k,\ell)}
&>& \frac{\s(H^{\ast}_0,k,\ell)}{2\s(H^{\ast}_1,k, \ell
)} \nonumber\\
& \geq& \frac{|\mathcal{S}_0||\mathcal{P}_s|}{2|\mathcal{S}_1||\mathcal{P}_{s'}|}\cdot \left(  1 +
\frac{4}{(3c-2)(3c-6)}
\right)^{\sum_{x=1}^{c-2} B(x) \binom{c-2 }{x}  } \nonumber \\
&=& \frac{3c}{2} \cdot \left(  1 +
\frac{4}{(3c-2)(3c-6)}
\right)^{\sum_{x=1}^{c-2} B(x)  \binom{c-2}{x}}
>1,
\end{eqnarray}
implying that the extremal hypergraphs have $\ell$-cover of size $c=c(k)$.

\subsection{The Case  $c  \geq q+1$}
Observe first, that for $r =2\ell - 1$ we have $q=1$ and equation
(\ref{anewjj1}) becomes
\begin{eqnarray} \label{anewjj1zz}
\frac{\s(H_0^{\ast}, P,k,\ell)}{\s(H_1^{\ast}, P',k,\ell)}
= \frac{1}{ 3^{B(1) (c-2)}} = 1,
\end{eqnarray}
since $B(1) = 0$ by (\ref{b737}).

Assume in the following that
$q = \lfloor r/\ell \rfloor \geq 2$. 
Notice that with (\ref{b737}) and (\ref{d737}) we have
$B(q) = 0$ and $B(q-1) = D(q-1) = \binom {n - \ell q}{ r - \ell q}$.
Then,  for $c \geq q+1 \geq 3$
 equality (\ref{anewjj1}) becomes with (\ref{klm1}):
\begin{eqnarray} \label{acnewjj1}
 \frac{\s(H_0^{\ast}, P,k,\ell)}{\s(H_1^{\ast}, P',k,\ell)}
&=& \frac{  \prod_{y=1}^{ q-1}
( 2 + 3y)^{2
B(y) \binom{c-2}{y}}  }
{ \prod_{z=1}^{min(c-2,q)} (  3 z)^{
B(z) \binom{c-2}{z}  
}
} \cdot \frac{ \prod_{x=1}^{q-2}
(4 +3x)^{A(x) \binom{c-2}{x}
 }
}{\prod_{x=1}^{ q-1}  (4 +3x)^{D(x) \binom{c-2}{x}
} } \nonumber \\
&=& \frac{  \prod_{y=1}^{ q-1}
( 2 + 3y)^{2
B(y) \binom{c-2}{y}}  }
{ \prod_{z=1}^{ q-1} (  3 z)^{
B(z) \binom{c-2}{z}  
}
} \cdot \frac{ \prod_{x=1}^{q-2}
(4 +3x)^{A(x) \binom{c-2}{x}
 }
}{\prod_{x=1}^{ q-1}  (4 +3x)^{D(x) \binom{c-2}{x}
} } \nonumber \\
&=& \left(
\prod_{x=1}^{q-2}  \frac{
( 2 + 3x)^{2
B(x) \binom{c-2}{x}}  }
{  (  3 x)^{
B(x) \binom{c-2}{x}  
} \cdot
(4 +3x)^{
B(x) \binom{c-2}{x}  
}
   }
  \right)  \nonumber \\
&& \times
\frac{ (3q-1)^{2B(q-1) \binom{c-2}{q-1}}  }{(3q-3)^{B(q-1)\binom{c-2}{q-1}}
 (3q+1)^{D(q-1) \binom{c-2}{q-1} }   } \nonumber \\
&=& \left(
\prod_{x=1}^{q-2}  \frac{
( 2 + 3x)^{2
B(x) \binom{c-2}{x}}  }
{  (  3 x)^{
B(x) \binom{c-2}{x} }
(4 +3x)^{
B(x) \binom{c-2}{x}  
  } } \right)   \nonumber \\
&& \times
\frac{ (3q-1)^{2B(q-1) \binom{c-2}{q-1}}  }{(3q-3)^{B(q-1) \binom{c-2}{q-1}}
 (3q+1)^{B(q-1) \binom{c-2}{q-1} }   } .
\end{eqnarray}
Clearly    we  have
$$
\frac{(2+3x)^2}{ 3x( 4 + 3x)}
> 1 \hspace{0.5cm} \mbox{   and   } \hspace{0.5cm}
\frac{(3q-1)^2}{(3q-3)(3q+1)}
> 1,
$$
 hence (\ref{acnewjj1}) becomes for $q \geq 2$ and $c \geq q+1$
\begin{eqnarray} \label{zabnewjj3}
&& \frac{\s(H_0^{\ast}, P,k,\ell)}{\s(H_1^{\ast}, P',k,\ell)}
\geq
\left( \frac{ (3q-1)^{2}}{(3q-3)
 (3q+1)} \right)^{B(q-1) \binom{c-2}{q-1} }  
> 1.
\end{eqnarray}
By (\ref{eq501}) and (\ref{eq502}), and
Lemma \ref{L5} we obtain from (\ref{zabnewjj3}) and
(\ref{anewjj1zz}), for any $q \geq 1$ and $r \geq 2\ell - 1$,
and, for $n$ sufficiently large,
\begin{eqnarray}\label{mona7}
 \frac{\kappa(H^{\ast}_0,k,\ell)}{\kappa(H^{\ast}_1,k,\ell)}
&>& \frac{\s(H^{\ast}_0,k,\ell)}{2\s(H^{\ast}_1,k,\ell
)}
 \geq  \frac{|\mathcal{S}_0||\mathcal{P}_s|}{2|\mathcal{S}_1||
\mathcal{P}_{s'}|}
=
 \frac{3c}{2}
>1,
\end{eqnarray}
hence also for $c \geq q+1$, and thus in all cases
the extremal hypergraphs have $\ell$-cover of size $c=c(k)$.

\end{document}